\documentclass[letterpaper, reqno]{amsart}

\usepackage[utf8]{inputenc}
\usepackage{amsmath,amsthm,amssymb}
\usepackage{mathtools}
\usepackage[left=1.5in,top=1.2in,right=1.5in,bottom=1.2in]{geometry}

\usepackage[colorlinks,
    linkcolor={red!50!black},
    citecolor={blue!50!black},
    urlcolor={blue!80!black}]{hyperref}

\usepackage{tikz}
\usepackage[arrow,matrix,curve,cmtip,ps]{xy}
\usepackage{pb-diagram}
\usepackage{pb-xy}

\usepackage{xcolor}

\newtheorem{proposition}{Proposition}[section]
\newtheorem{theorem}[proposition]{Theorem}
\newtheorem{corollary}[proposition]{Corollary}
\newtheorem{lemma}[proposition]{Lemma}

\newtheorem{thmx}{Theorem}

\newtheorem{thmxprime}{Theorem}

\newtheorem*{thm:distant-knots-cobordisms}{Theorem~\ref{thm:distant-knots-cobordisms}}
\theoremstyle{definition}
\newtheorem{definition}[proposition]{Definition}

\theoremstyle{remark}
\newtheorem{remark}[proposition]{Remark}

\newtheorem*{remark*}{Remark}
\newtheorem*{ackn}{Acknowledgements}
\newtheorem{const}[proposition]{Construction}

\newcommand{\Diff}{\text{Diff}}
\newcommand{\lk}{\operatorname{lk}}

\newcommand{\bdry}{\partial}
\renewcommand{\top}{\text{top}}

\newcommand{\Q}{\mathbb{Q}}
\newcommand{\N}{\mathbb{N}}

\newcommand{\CC}{\mathbb{C}}

\newcommand{\Z}{\mathbb{Z}}

\newcommand{\lnk}{\operatorname{lk}}
\newcommand{\Id}{\operatorname{Id}}
\newcommand{\SO}{\operatorname{SO}}

\newcommand{\into}{\hookrightarrow}

\begin{document}
\title{Concordance of knots in $S^1\times S^2$}

\author{Christopher W.\ Davis}
\address{Department of Mathematics, University of Wisconsin--Eau Claire}
\email{daviscw@uwec.edu}
\urladdr{people.uwec.edu/daviscw}

\author{Matthias Nagel}
\address{Department of Mathematics \& Statistics McMaster University, Canada}
\email{nagel@cirget.ca}
\urladdr{http://nagel.io/}

\author{JungHwan Park}
\address{Max-Planck-Institut f\"{u}r Mathematik}
\email{jp35@mpim-bonn.mpg.de }
\urladdr{http://people.mpim-bonn.mpg.de/jp35}

\author{Arunima Ray}
\address{Max-Planck-Institut f\"{u}r Mathematik}
\email{aruray@mpim-bonn.mpg.de }
\urladdr{}

\date{\today}


\begin{abstract}
We establish a number of results about smooth and topological concordance of knots in $S^1\times S^2$. The \emph{winding number} of a knot in $S^1\times S^2$ is defined to be its class in $H_1(S^1\times S^2;\Z)\cong \Z$. We show that there is a unique smooth concordance class of knots with winding number one. This improves the corresponding result of Friedl--Nagel--Orson--Powell in the topological category. We say a knot in $S^1\times S^2$ is slice (resp.\ topologically slice) if it bounds a smooth (resp.\ locally flat) disk in $D^2\times S^2$. We show that there are infinitely many topological concordance classes of non-slice knots, and moreover, for any winding number other than $\pm 1$, there are infinitely many topological concordance classes even within the collection of slice knots. Additionally we demonstrate the distinction between the smooth and topological categories by constructing infinite families of slice knots that are topologically but not smoothly concordant, as well as non-slice knots that are topologically slice and topologically concordant, but not smoothly concordant.
\end{abstract}

\maketitle
\section{Introduction}
A knot $K$ in a closed, oriented $3$-manifold $Y$ is an oriented, smooth embedding $K\colon S^1\hookrightarrow Y$. As is customary, often we will conflate a knot and its image. Two knots in $Y$ are said to be \emph{concordant (resp.\ topologically concordant)} if they cobound a smooth (resp.\ locally flat) annulus in $Y\times I$ (see Section~\ref{sec:prelims} for more precise definitions). Note that a smooth annulus is locally flat and thus, concordant knots are topologically concordant. Topological concordance of knots in general $3$-manifolds was studied in~\cite{FNOP16}. In this paper we focus on concordance and topological concordance of knots in $S^1\times S^2$. Fix an orientation of $S^1\times S^2$.  Then any knot $K$ in $S^1\times S^2$ corresponds to an element of the set of oriented free homotopy classes of loops in $S^1\times S^2$. The latter is identified with $H_1(S^1\times S^2;\Z)\cong \Z$, and thus, any knot $K$ corresponds to a well-defined integer, called the \emph{winding number of~$K$}, denoted~$w(K)$. Winding number is preserved under concordance in either category, and thus, we are interested in concordance classes of knots with a fixed winding number. Let~$H$ denote the knot given by $S^1\times \{pt\}\subseteq S^1\times S^2$ with winding number one; we call $H$ the \emph{Hopf knot}.

In~\cite[Theorem 1.4 \& 1.6]{FNOP16}, it was shown that there are infinitely many topological concordance classes of knots in $S^1\times S^2$ with winding number $w= 0, \pm 2$. In contrast, they showed that any winding number $1$ knot is topologically concordant to the Hopf knot (and consequently, any winding number $-1$ knot is topologically concordant to the reverse of $H$); this is called the \emph{topological concordance lightbulb theorem}. We expand on these results to show that there are infinitely many topological concordance classes of knots in $S^1\times S^2$ with any winding number other than $\pm 1$. For this outstanding case, we promote the result from~\cite{FNOP16} to the smooth category, as follows.

\begin{thmx}[Smooth concordance lightbulb theorem]\label{thm:thmA}
Any winding number $1$ knot in $S^1\times S^2$ is smoothly concordant to the Hopf knot $H$.
Consequently any winding number $-1$ knot is smoothly concordant to the reverse of the Hopf knot.
\end{thmx}

\begin{figure}[htb]
\begin{tikzpicture}
\node at (0,0) {\includegraphics[height=3cm]{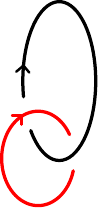}};
\node at (-1,-1) {$H$};
\node at (0.9,1) {$0$};
\end{tikzpicture}
\caption{The knot~$S^1\times\{p\} \subset S^1\times S^2$ with winding number $1$ is the \emph{Hopf knot}~$H$. Note that any knot with winding number~$1$ is homotopic to the Hopf knot. We will use the term `Hopf knot' to refer to any knot ambiently isotopic to $H$.}\label{fig:Hopf-knot}
\end{figure}

In fact, our proof applies to a broader notion of concordance, called \emph{$\Q$-concordance}, defined in Section~\ref{sec:prelims}, as follows.

\begin{thmxprime}\label{thm:thmAprime}
Any knot in $S^1\times S^2$ with nonzero winding number is $\Q$-concordant to $H$.
\end{thmxprime}

Of course, Theorems~\ref{thm:thmA} and~\ref{thm:thmAprime} would not be surprising if all knots in $S^1\times S^2$ were isotopic to one another. However, this is not the case; we show in Section~\ref{sec:prelims} that there are indeed infinitely many isotopy classes of knots with any fixed winding number in $S^1\times S^2$.

Recall that a knot in $S^3$ is said to be \emph{slice (resp.\ topologically slice)} if it bounds a smooth (resp.\ locally flat) disk in $D^4$, or equivalently, if it is smoothly (resp.\ topologically) concordant to the unknot.
We say that a knot~$K$ in $S^1\times S^2$ is \emph{slice (resp.\ topologically slice)} if it bounds a smoothly (resp.\ locally flat) embedded disk in $D^2\times S^2$. Note that any slice knot is topologically slice. If a knot is concordant to a slice knot, it is of course slice itself. With our definition it is tempting to say that a knot in $S^1\times S^2$ is slice exactly if it is concordant to $H$. However, note that the $(p,1)$ cable of $H$, for any $p\neq 0$, bounds a smooth disk in $D^2\times S^2$ constructed by taking $p$ copies of the disk bounded by $H$, banding them together as directed by the cabling pattern, and pushing the interior of the bands into $D^2\times S^2$.
\begin{figure}[htb]
\begin{tikzpicture}
\node at (0,0) {\includegraphics[width=4cm]{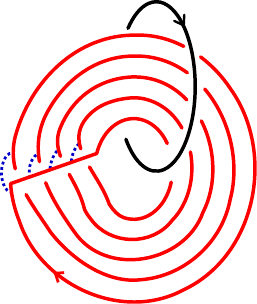}};
\node at (1,2.3) {$0$};
\end{tikzpicture}
\caption{The $(5,1)$-cable~$H_{5,1}$ of the Hopf knot,
together with a sequence of band moves reducing~$H_{5,1}$ to
$5$ parallel copies of the Hopf knot.}\label{fig:Hopf-knot-cable}
\end{figure}
On the other hand, these cables are distinct even up to topological concordance since they have distinct winding numbers. We show in the following theorem that winding number is far from being the only obstruction for slice knots to be concordant. 

\begin{thmx}\label{thm:thmB}
For any $w \in \Z$ with $w \neq \pm1$, there exists an infinite family of slice knots with winding number~$w$ which are pairwise distinct in topological concordance.
\end{thmx}

\noindent{We also establish that there exist knots that are not topologically slice.}

\begin{thmx}\label{thm:thmC}
For any $w \in \mathbb{Z}$ with $w \neq \pm 1, \pm 3$, there exists an infinite family of winding number $w$ knots, none of which is topologically slice, and which are pairwise distinct up to topological concordance.
\end{thmx}

Above note that as a consequence of Theorem \ref{thm:thmA}, any winding number $\pm 1$ knot is topologically slice. In contrast, we believe our inability to show the above for winding number~$\pm 3$ knots is a limitation of our techniques, rather than an actual outstanding case. We are unable to even produce one example of a non-slice knot of winding number~$\pm 3$.

Lastly, we study the disparity between the topological and smooth categories, by establishing analogues of Theorem~\ref{thm:thmB} and Theorem~\ref{thm:thmC} using families of topologically concordant knots.

\begin{thmxprime}\label{thm:thmBprime}
For any $w \in \mathbb{Z}$ with $w \neq \pm1$, there exists an infinite family of slice knots with winding number $w$ which are topologically concordant to one another but pairwise distinct in smooth concordance.
\end{thmxprime}

\begin{thmxprime}\label{thm:thmCprime}
 For any $w \in 2\mathbb{Z}$, there exists an infinite family of non-slice knots with winding number $w$ which are topologically slice, topologically concordant to one another, but pairwise distinct in smooth concordance.
\end{thmxprime}

For odd winding numbers, we expect that there is an argument, similar to what we use in the proof of Theorem~\ref{thm:thmC}, using $d$-invariants~\cite{OzSz03}. The main limitation appears to be a lack of computability.

\subsection*{Outline}In Section~\ref{sec:prelims}, we give precise definitions as well as some results about knots in $S^1\times S^2$ that are contained in a ball, called \emph{distant knots}, and isotopy classes of knots. In Section~\ref{sec:concordance-to-Hopf}, we prove Theorems~\ref{thm:thmA} and~\ref{thm:thmAprime}. This consists of an explicit construction of a $4$-manifold which we then show is diffeomorphic to $S^1\times S^2\times I$. Section~\ref{sec:slice-knots} gives the proof of Theorem~\ref{thm:thmB}, using covering links and our previous results about distant knots. We address non-slice knots with even winding numbers in Section~\ref{sec:nonslice-even} and those with odd winding numbers in Section~\ref{sec:nonslice-odd}; together the results of these sections give the proof of Theorem \ref{thm:thmC}. The techniques of Section~\ref{sec:nonslice-even} are quite direct, whereas Section~\ref{sec:nonslice-odd} involves a technical theorem about Casson-Gordon invariants. Lastly, we address the difference between the smooth and topological categories in Section~\ref{sec:smooth-vs-top}, where we prove Theorems \ref{thm:thmBprime} and \ref{thm:thmCprime}, which involves appropriate modifications of the proofs of Theorems~\ref{thm:thmB} and~\ref{thm:thmC}.

\begin{ackn}
We thank Stefan Friedl and Mark Powell for helpful discussions. We are also grateful to the anonymous referee for thoughtful suggestions. This project was started and completed while the authors were participating in the Junior Hausdorff Trimester Program in Topology at the Hausdorff Research Institute for Mathematics in the fall of 2016. We thank HIM for a stimulating research atmosphere. MN is supported by a CIRGET postdoctoral fellowship.
\end{ackn}

\begin{remark*}
Lisa Piccirillo informed us in the spring of 2017 that she has an independent unpublished proof of Theorem~\ref{thm:thmA}. A similar proof to hers was independently obtained by Yildiz~\cite{Yil17}, whose paper appeared on the arXiv immediately prior to this one.
\end{remark*}

\section{Preliminaries}\label{sec:prelims}
We begin this section by formalizing the definitions we use throughout the paper.

A \emph{knot}~$K \colon \gamma \into Y$ is a smooth embedding
in a closed, oriented $3$-manifold~$Y$ of
an oriented $1$-manifold~$\gamma$ diffeomorphic to $S^1$. We identify two such parametrizations
$K\colon \gamma \hookrightarrow Y$ and $K' \colon \gamma' \hookrightarrow Y$
if there exists an orientation preserving diffeomorphism~$\gamma \to \gamma'$
mapping between the two embeddings.

We write $S^1$ for the unit circle oriented as the boundary of the postively oriented unit disk.
For an oriented manifold $X$, the manifold~$\overline X$ denotes the same manifold with the opposite orientation.
Given a knot $K\colon S^1\hookrightarrow Y$, the \emph{reverse knot} $rK\colon\overline{S^1}\hookrightarrow Y$ is the same function as $K$, but with the opposite orientation on $S^1$. Similarly, the \emph{mirror} knot $\overline{K}\colon S^1\hookrightarrow \overline{Y}$ is the same function as $K$, but with the opposite orientation on $Y$.
We write $-K$ for the knot~$r \overline K$.

Recall that two knots $K$ and $J$ in any closed, oriented $3$-manifold $Y$ are said to be \emph{isotopic} if they are homotopic through smooth embeddings; they are said to be \emph{ambient isotopic} if there exists a $1$-parameter family of diffeomorphisms $F_t:Y\to Y$ such that $F_0=\Id$ and $F_1\circ K=J$; and they are said to be \emph{equivalent} if there is an orientation preserving diffeomorphism $\phi: Y\to Y$ such that $\phi\circ K=J$. By the isotopy extension theorem, the first two notions are equivalent for any $Y$. Indeed, for $Y=S^3$, any orientation preserving diffeomorphism is smoothly isotopic to the identity map~\cite{Cerf68}, and thus, the three notions are identical. There is a vast body of literature concerning knots in $S^3$, up to any of the above notions of equivalence. Note that there exist $3$-manifolds where not all orientation preserving diffeomorphisms are smoothly isotopic to the identity, e.g.\ in $S^1\times S^2$, the Gluck twist, which gives the $S^2$ factor a full rotation as one travels along the $S^1$ factor, is not isotopic to the identity by \cite{Gluck62}. Thus, for knots in $S^1\times S^2$, there is a need to be careful in specifying which $3$-dimensional notion of equivalence is being considered. In this paper, we will use the notion of ambient isotopy, i.e.\ ambient isotopic knots will be considered to be identical. As noted above, this is the same as considering knots up to isotopy.

Given an oriented manifold $X$, consider the manifold $X\times I$.
The orientation on $I$ induces an orientation on $X\times I$, which in turn induces the boundary orientation on $\partial^+ (X\times I)\coloneqq X\times \{1\}$ and
$\partial^- (X\times I) \coloneqq X\times \overline{\{0\}}$, where $\{pt\}$ denotes a point with the positive orientation and $\overline{\{pt\}}$ has the negative orientation. Under these conventions, we have the identifications $\partial^+ (X\times I)=X$ and
$\partial^- (X\times I)=\overline{X}$. Additionally, given a function $f\colon X\times I\rightarrow Y$,
we write $\partial^+f \coloneqq f\big|_{\partial^+X} \colon X \to Y$,
and $\partial^-f\coloneqq f\big|_{\partial^-X} \colon \overline X \to Y$.

We say that two knots $K_0$ and $K_1$ in a closed, connected, oriented $3$-manifold $Y$ are \emph{concordant (resp.\ topologically concordant)} if there exists a smooth (resp.\ locally flat) properly embedded annulus~$A\colon S^1\times I \hookrightarrow Y\times I$ such that
\[
\partial^+ A \sqcup \partial^- A= K_1 \sqcup -K_0 \subset Y \sqcup \overline Y;
\]
i.e.\ $\partial^+ A\colon S^1\rightarrow \partial^+(Y\times I)=Y$ is the knot $K_1$ and $\partial^- A\colon \overline{S^1}\rightarrow \partial^-(Y\times I)=\overline{Y}$ is the knot~$- K_0$.

Clearly, winding number is preserved under isotopy and ambient isotopy. More generally, if $K$ and $J$ are homotopic knots (i.e.\  homotopic as maps to $S^1\times S^2$), $w(K)=w(J)$. Since a concordance yields a homotopy, winding number is also preserved under concordance in either category. We will draw pictures of knots in the standard surgery diagram for $S^1\times S^2$, e.g.\ in Figure~\ref{fig:Hopf-knot}. Recall that
the orientation on $S^3$ (resp.\ $B^4$) gives a well-defined orientation on the manifolds obtained as a realization of a surgery (resp.\ Kirby) diagram.
In the standard diagram in Figure~\ref{fig:Hopf-knot},
we orient the $0$-framed surgery curve, which gives an identification of $H_1(S^1\times S^2;\Z)$ with $\Z$, generated by the meridian with linking number one. Note that this shows that the Hopf knot has winding number one, as desired.

Let $M$ and $N$ be closed, oriented $3$-manifolds. For any abelian group~$R$, an \emph{$R$-homology cobordism} from $M$ to $N$ is a triple $(W,\phi^+, \phi^-)$ where $W$ is a compact, smooth $4$-manifold such that
\begin{enumerate}
\item $\partial W = \partial^- W \sqcup \partial^+ W$,
\item $\phi^+ \colon \partial^+ W \to M$ is an orientation preserving diffeomorphism,
\item $\phi^- \colon \partial^- W \to N$ is an orientation reversing diffeomorphism, and
\item the induced maps $H_*(\partial^\pm W;R)\to H_*(W;R)$ are isomorphisms.
\end{enumerate}
We will use the term \emph{homology cobordism} for a $\Z$-homology cobordism.

\begin{definition}
Let $R$ be an abelian group.  We say that the knots $K_1\colon S^1\hookrightarrow M$ and $K_0\colon S^1 \hookrightarrow N$, for closed, oriented $3$-manifolds $M$ and $N$, are \emph{$R$-homology concordant} or just \emph{$R$-concordant} if there exists an $R$--homology cobordism $(W,\phi^+,\phi^-)$ from $M$ to $N$ and a smooth, proper embedding $A\colon S^1\times I\into W$ such that $\phi^+\circ \partial^+ A = K_1$ and $\phi^-\circ \partial^- A = rK_0$. If $R=\Z$, we simply say that $K$ and $J$ are \emph{homology concordant}.
\end{definition}

\begin{remark}
Above, the reader may have expected the requirement that $\phi^-\circ \partial^- A = r\overline{K_0}$ instead of $rK_0$. This is accounted for by the fact that $\phi^{-1}$ is orientation reversing.
\end{remark}

We note that a (smooth) concordance between two knots in a $3$-manifold $Y$ is simply a smooth, proper embedding of an annulus in the homology cobordism $(Y\times I, \Id, \overline{\Id})$. That is, not only do we fix the homology cobordism, but also the diffeomorphisms on the boundary components.
The map~$\Id$ is the orientation preserving map $\partial^+(Y\times I) = Y \rightarrow Y$ given by $\Id(y)=y$ and $\overline{\Id}$ is the orientation reversing map~$\partial^- (Y\times I) =\overline Y\rightarrow Y$ given by $\overline{\Id}(y)=y$.

\begin{proposition}\label{prop:op-diffeo-concordant}
If $K$ and $J$ are knots in a compact, oriented, connected $3$-manifold $Y$ and $\phi\colon Y\to Y$ is an orientation preserving diffeomorphism such that $\phi\circ J=K$, then $K$ and $J$ are homology concordant.
\end{proposition}

\begin{proof}
The desired homology cobordism is $(Y \times I; \phi, \overline{\Id})$ where $J\times\Id_{I}$ is the required concordance. We have the orientation preserving map  $\phi$ given by
\begin{align*}
\phi\colon \partial^+(Y\times I) = Y&\rightarrow Y\\
y&\mapsto \phi(y)
\end{align*}
and as before, $\overline{\Id}$ is the orientation reversing map
\begin{align*}
\overline\Id\colon  \partial^-(Y\times I) = \overline Y&\rightarrow Y\\
y&\mapsto y
\end{align*}
As needed, we see that
\[\phi\circ \partial^+(J\times\Id_{I}) = \phi(J)=K\]
 and
 \[\overline{\Id}\circ \partial^-(J\times\Id_{I}) = \overline{\Id}\circ r\overline{J}=rJ. \qedhere
 \]\end{proof}

\begin{remark}\label{rem:reverse}
In particular, consider the diffeomorphism
\begin{align*}
\phi:S^1\times S^2&\to S^1\times S^2\\
(s,t) &\mapsto (\overline{s}, \overline{t})
\end{align*}
where we reflect along a hyperplane in each factor. This reverses the orientations of the $S^1$ and $S^2$ factors, and thus, $\phi$ is orientation preserving. In particular, $\phi$ sends
the Hopf knot to the reverse of the Hopf knot, and so the Hopf knot is homology concordant to its reverse by the previous proposition.
\end{remark}

For any commutative ring~$R$, since $S^1\times S^2$ and $S^1$ have fixed orientations, we get an identification $R = H_1(S^1 \times S^2; R)$. The \emph{$R$-winding number} of a knot $K$ is defined to be the class $[K] \in H_1(S^1 \times S^2; R) = R$. From Remark \ref{rem:reverse}, we see that two knots~$K$ and $J$ may be homology concordant and not topologically concordant, since in particular, the Hopf knot and its reverse have different winding numbers. If $K_0$ and $K_1$ are $R$-homology concordant in the $R$-homology cobordism $(W; \phi^+,\phi^-)$, then the isomorphism
\[
(\phi^+_*)\circ (\phi^-_*)^{-1}:H_1(S^1\times S^2;R)\to H_1(S^1\times S^2;R)
\]
sends the class $[K_0] \in H_1(S^1 \times S^2; R) = R$ to the class $[K_1]$. That is, there is an isomorphism of the underlying groups~$R\xrightarrow{\cong} R$ sending the $R$-winding number of $K_0$ to that of~$K_1$. The fact that concordant knots have the same winding number corresponds to the fact that $(\Id_*)\circ (\overline\Id_*)^{-1}$ is the identity map on $\Z$, whereas for a homology cobordism, a general diffeomorphism might induce the map on $\Z$ given by multiplying by $-1$, such as the map $\phi$ in Remark~\ref{rem:reverse}. As a sanity check going forward, we see that Theorem~\ref{thm:thmAprime}\, requires the omission of winding number $0$ knots in $S^1\times S^2$, since the $0$ map is obviously not an isomorphism on $\Q$. However, there do exist isomorphisms of $\Q$ taking $1$ to any non-zero integer of our choice, and thus, winding number considerations at least do not contradict Theorem~\ref{thm:thmAprime}.

\subsection{Distant knots}
From the identification~$S^3\# S^1 \times S^2 = S^1 \times S^2$ we obtain
a inclusion~$\Phi$ from the set of knots in $S^3$ to the set of
knots in $S^1 \times S^2$. The image of this map consists of the knots in
$S^1\times S^2$ which may be isotoped to be disjoint from one fiber~$\{\text{pt}\} \times S^2$, or equivalently, knots that are contained in a ball. We call such knots \emph{distant knots}. Below we prove that $\Phi$ induces a well-defined injective map on concordance classes.

\begin{theorem}\label{thm:distant-knots-cobordisms} Let $S$ be an abstract (possibly non-orientable) surface with two boundary components. Two knots $K$ and $J$ in $S^3$ cobound a smooth surface homeomorphic to $S$ in $S^3\times I$ if and only if  $\Phi(K)$ and $\Phi(J)$ cobound a smooth surface homeomorphic to $S$ in $S^1\times S^2\times I$.

In particular, $K$ and $J$ are smoothly concordant in $S^3 \times I$ if and only if $\Phi(K)$ and $\Phi(J)$ are smoothly concordant in $S^1\times S^2\times I$.

The same is true in the topological category.
\end{theorem}

A result similar to Theorem~\ref{thm:distant-knots-cobordisms} was shown in~\cite{NOPP17}, using completely different techniques, by Powell, Orson, and the second and third authors. Their method addresses only concordance of distant knots (rather than possibly non-orientable surfaces of non-zero genus cobounded by them), but applies to all closed, oriented $3$-manifolds.

\begin{proof}[Proof of Theorem~\ref{thm:distant-knots-cobordisms}]
Suppose that $K$ and $J$ cobound a smooth surface~$F\cong S$ in $S^3\times I$.  Pick a point $p\in S^3$ disjoint from $K$ and $J$. By an isotopy, we assume that $F$ is disjoint from $\{p\}\times I$.  Thus, for a small neighborhood $B$ of $p$, $F$ is disjoint from $B\times I$.  Modify $S^3\times I$ by cutting out $B\times I$ and gluing in $(S^1\times S^2-B')\times I$ for some small $3$-ball $B'$ in $S^1\times S^2$.  The resulting $4$-manifold is $S^1\times S^2\times I$.  The image of $F$ in this $4$-manifold is a smooth surface cobounded by $\Phi(K)$ and $\Phi(J)$.

Conversely, suppose that $\Phi(K)$ and $\Phi(J)$ cobound a smooth surface~$F\cong S$ in $S^1\times S^2\times I$. Let $C$ denote the cylinder $K\times I \subset S^3\times I$. Add a $1$-handle to $S^3\times\{1\}$ away from $K$. The resulting $4$-manifold~$W_0$ has boundary $-S^3\sqcup S^1\times S^2$ and $\bdry C = K\sqcup \Phi(K)$. Attach $(S^1\times S^2\times I, F)$ to $(W_0,C)$ along their common boundary components to obtain a $4$-manifold $W_1$. Now we cancel the $1$-handle by attaching a $2$-handle to the positive boundary of $W_1$, so that the resulting $4$-manifold is $S^3\times I$. In this $4$--manifold $C\cup F$ is a smooth surface cobounded by $K$ and $J$.

An identical proof works in the topological category.
\end{proof}

\noindent We give a similar result for surfaces bounded by distant knots in $D^2\times S^2$.

\begin{proposition}\label{prop:distant-knots-slices}
Let $K$ be a knot in $S^3$ and $S$ be an abstract (possibly non-orientable) surface with a single boundary component. The distant knot $\Phi(K)\subset S^1\times S^2$  bounds a smoothly embedded surface homeomorphic to $S$ in $D^2\times S^2$ if and only if $K$  bounds a smoothly embedded surface homeomorphic to $S$ in $D^4$.

 The same is true in the topological category.
\end{proposition}

\begin{proof}
Suppose $K$ bounds $F\cong S$ in $D^4$. Add a $2$-handle to $D^4$ along a $0$-framed unknot unlinked from $K$. The resulting $4$-manifold is $D^2\times S^2$ and $F\cong S$ is now a surface bounded by $\Phi(K)$.

Conversely, suppose that $F\cong S$ is a surface bounded by $\Phi(K)$ in $D^2 \times S^2$. Since $\Phi(K)$ is contained in a $3$-ball, we isotope it away from $\{1\} \times S^2$ and add a $3$-handle along $\{1\} \times S^2$ in $S^1\times S^2 = \bdry (D^2\times S^2)$ to get $D^4$. The image of $\Phi(K)$ is precisely $K$, and it bounds $F$.
\end{proof}

Recall that for a knot $K\subset S^3$, the \emph{$4$-genus~$g_4(K)$} is the minimum genus of smooth connected oriented surfaces in $D^4$ bounded by $K$. The \emph{non-orientable $4$-genus~$\gamma_4(K)$} of $K$ is the minimum genus of all (possibly non-orientable) surfaces bounded by $K$. We similarly define the \emph{topological $4$-genus~$g^{\top}_4(K)$}, and \emph{topological non-orientable 4-genus~$\gamma^{\top}_4(K)$} by allowing surfaces that are locally flat but may not be smooth. For a knot $K\subset S^1\times S^2$, we similarly define $g_4(K)$, $\gamma_4(K)$, $g^{\top}_4(K)$,
and $\gamma^{\top}_4(K)$ by considering surfaces in $D^2\times S^2$. This gives us the following version of~Proposition~\ref{prop:distant-knots-slices}.

\begin{corollary}\label{cor:slice-distant-knots}
For any knot $K\subset S^3$,
$g_4(K) = g_4(\Phi(K))$, $g^{\top}_4(K) = g^{\top}_4(\Phi(K))$,
$\gamma_4(K) = \gamma_4(\Phi(K))$, and $\gamma^{\top}_4(K) = \gamma^{\top}_4(\Phi(K))$.
\end{corollary}

Lastly we note that not all winding number zero knots in $S^1\times S^2$ are concordant to distant knots. For example, consider the knot shown in Figure~\ref{fig:not-distant}. Lift to the universal cover $\mathbb{R}\times S^2$ of $S^1\times S^2$. The lift of the knot has infinitely many components. Since $\mathbb{R}\times S^2\subseteq S^3$ we can consider the usual linking number of these components, and we see that consecutive components have linking number one. In contrast, a distant knot must lift to unliked components since it lies in a ball. Since linking number is invariant under concordance, this completes our argument.

\begin{figure}[htb]
\begin{tikzpicture}
\node at (0,0) {\includegraphics[height=5cm]{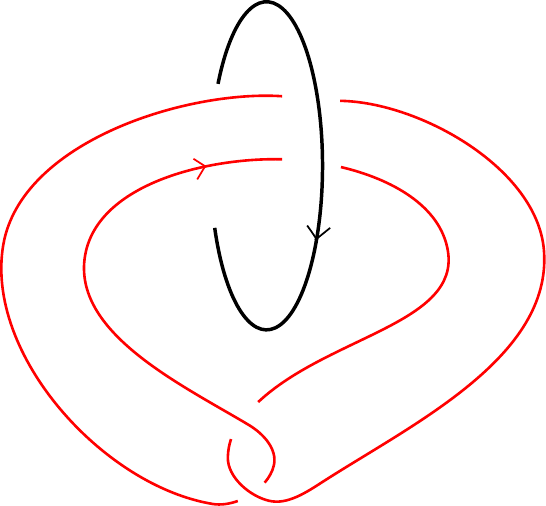}};
\node at (0.5,2) {$0$};
\end{tikzpicture}
\caption{A winding number $0$ knot that is not concordant to a distant knot.}\label{fig:not-distant}
\end{figure}

\subsection{Isotopy classes of knots}
In this brief interlude we address isotopy classes of knots in $S^1\times S^2$. The \emph{geometric winding number} of a knot in $S^1\times S^2$ is the minimum number of times (counted without sign) that it intersects a sphere isotopic to $\{pt\}\times S^2$. Note that this quantity is preserved under isotopy of knots.

We sketch a proof which is an extremely mild generalization of an argument from~\cite[Section~4]{Livingston85}. In that paper, Livingston shows that two given winding number $1$ knots in $S^1\times S^2$ are not isotopic since they have geometric winding number $3$ and $5$ respectively. However, as we show below, the same approach applies to distinguish infinitely many knots with any fixed winding number.

\begin{proposition}[\cite{Livingston85}]
There exist infinitely many isotopy classes of knots with any fixed winding number $w\in\Z$.
\end{proposition}

\begin{proof}[Proof sketch]
Fix a winding number $w\in\N\cup \{0\}$. Consider the family of knots $K_{i,j}\subset S^1\times S^2$ shown in Figure~\ref{fig:nonisotopic-knots} with $j-i=w$. We will show that the geometric winding number of $K_{i,j}$ is $i+j$. Since geometric winding number is preserved under isotopy and winding number changes sign under orientation reversal, this will complete the proof.

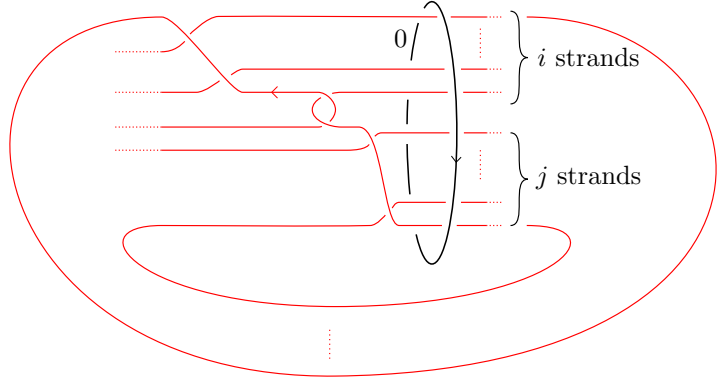
\begin{figure}[htb]
\begin{tikzpicture}
\node at (0,0) {\includegraphics[height=5cm]{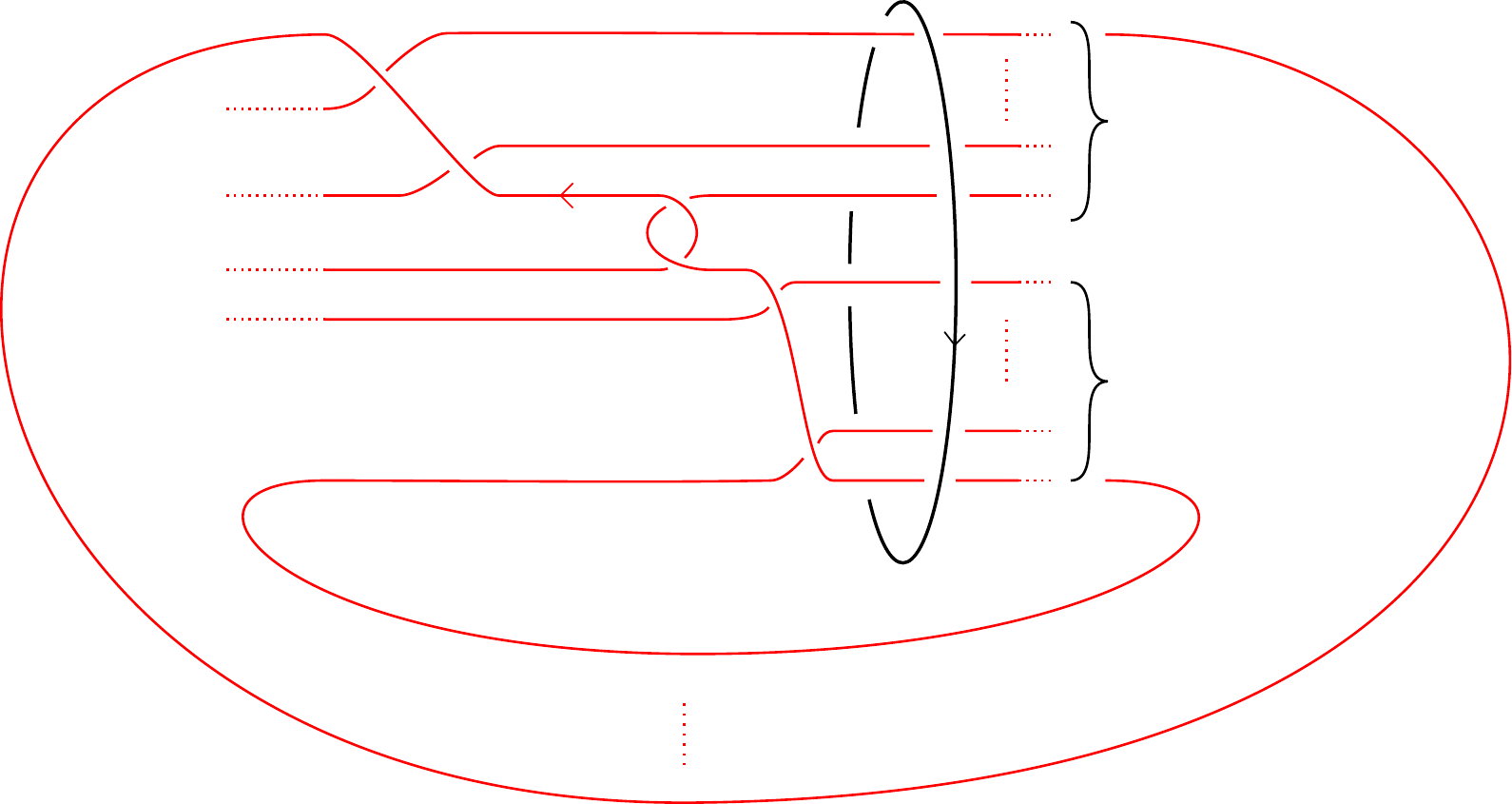}};
\node at (0.5,2) {$0$};
\node at (3,1.75) {$i$ strands};
\node at (3,0.15) {$j$ strands};
\end{tikzpicture}
\caption{The family of knots $K_{i,j}$ in $S^1\times S^2$.}\label{fig:nonisotopic-knots}
\end{figure}
Livingston~\cite[Section~2]{Livingston85} shows that for a
knot~$K\subset S^1\times S^2$, if $K$ intersects two spheres~$\{pt\}\times S^2$
and $S$ in $n$ and strictly fewer than $n$ points respectively, then we may assume that $\{pt\}\times S^2$ and $S$ are disjoint.
In Figure~\ref{fig:nonisotopic-knots}, we see that $K_{i,j}$ intersects the sphere $\{pt\}\times S^2$ in $i+j$ points. Suppose there exists another sphere $S$, isotopic to $\{pt\}\times S^2$, such that it intersects $K_{i,j}$ in fewer points. By the result mentioned earlier, we can assume $S$ to be disjoint from $\{pt\}\times S^2$. Moreover, since any essential sphere in $S^1\times S^2$ is isotopic to $\{pt\}\times S^2$, the new $S$ is still isotopic to $\{pt\}\times S^2$. Cut along $\{pt\}\times S^2$ to obtain $I\times S^2$, within which $K_{i,j}$ appears as $i+j$ arcs and the image of $S$ is an essential sphere. It is easy to see that $i+j-2$ of these arcs traverse from $\{0\}\times S^2$ to $\{1\}\times S^2$ and thus, the image of $S$ in $I\times S^2$, must intersect each of these $i+j-2$ arcs. The remaining two arcs form a clasp, and thus, $S$ cannot split them. This shows that $S$ must intersect $K_{i,j}$ in at least $i+j$ points, which is a contradiction.
\end{proof}


Note that for winding numbers other than $\pm 1$ the result above also follows from our Theorem~\ref{thm:thmB} or~\ref{thm:thmC}. In contrast, the result indicates that there is indeed some content to our Theorem~\ref{thm:thmA}.

\section{Concordance to the Hopf knot}\label{sec:concordance-to-Hopf}

In this section, we prove Theorems~\ref{thm:thmA} and \ref{thm:thmAprime}. Our goal for now will be to show that any winding number $1$ knot is smoothly concordant to the Hopf knot $H$. Reversing the orientation of such a concordance will imply that any knot with winding number $-1$ is smoothly concordant to the reverse of the Hopf knot, $rH$, finishing the proof of Theorem \ref{thm:thmA}.

Given an arbitrary knot $K$ in $S^1\times S^2$, after an isotopy, we may draw it as the closure of a tangle~$P$, as shown in Figure~\ref{fig:diffeomorphic-knot}. Recall that we have fixed an identification of $S^1 \times S^2$ as the realization of the diagram in Figure~\ref{fig:diffeomorphic-knot}.
\begin{figure}[htb]
\begin{tikzpicture}
\node at (0,0) {\includegraphics[width=6cm]{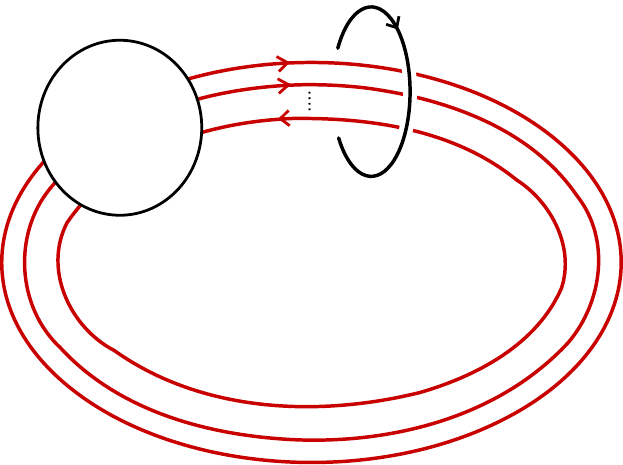}};
\node at (-1.85,1) {\Huge $P$};
\node at (0.7,0.3) {$0$};
\node at (-1.8,-0.9) {$K$};
\end{tikzpicture}
\caption{A generic knot $K$ in $S^1 \times S^2$ given as the closure of a tangle $P$.}\label{fig:diffeomorphic-knot}
\end{figure}
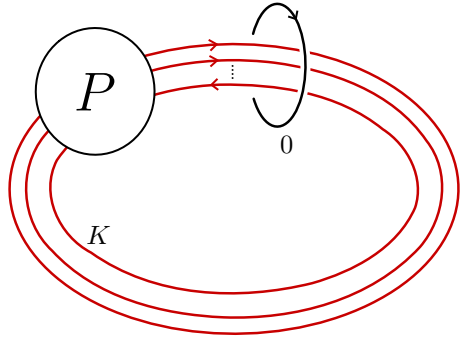
We will construct and study a cobordism $W$ from $S^1\times S^2$ to itself, which supports a smooth concordance from $K$ to the Hopf knot.

Recall that a $0$-framed meridian of a knot is called a \emph{helper circle}, since by sliding the knot over such a helper circle, we can change any crossing of the knot. Also, by such slides, we can separate the knot from any other surgery curve.
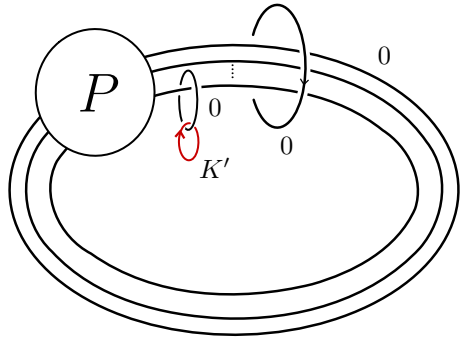
\begin{figure}[htb]
\begin{tikzpicture}
\node at (0,0) {\includegraphics[width=6cm]{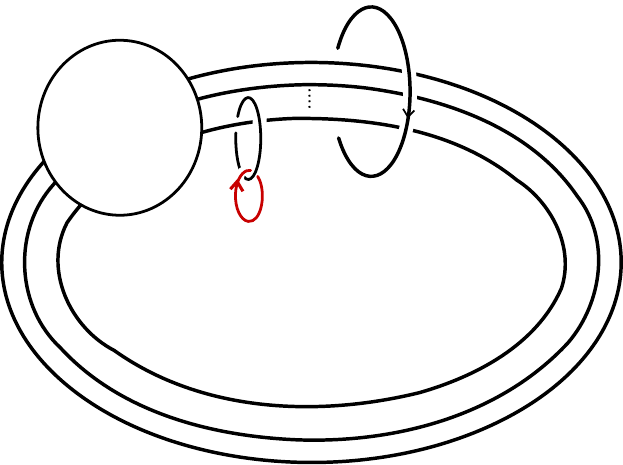}};
\node at (-1.8,1) {\Huge $P$};
\node at (2,1.5) {$0$};
\node at (0.7,0.3) {$0$};
\node at (-0.25,0.8) {$0$};
\node at (-0.25,0.0) {$K'$};
\end{tikzpicture}
\caption{A $3$-manifold $M \cong S^1 \times S^2$ containing a knot $K'$.}
\label{fig:diffeomorphic-knot-1}
\end{figure}
\begin{lemma}\label{lem:modify-knot}
Let $K'$ and $M$ be as shown in Figure~\ref{fig:diffeomorphic-knot-1}. There is an orientation preserving diffeomorphism
$\psi\colon M\to S^1\times S^2$ sending $K'$ to the knot $K$ depicted in Figure~\ref{fig:diffeomorphic-knot}.
\end{lemma}
\begin{proof}
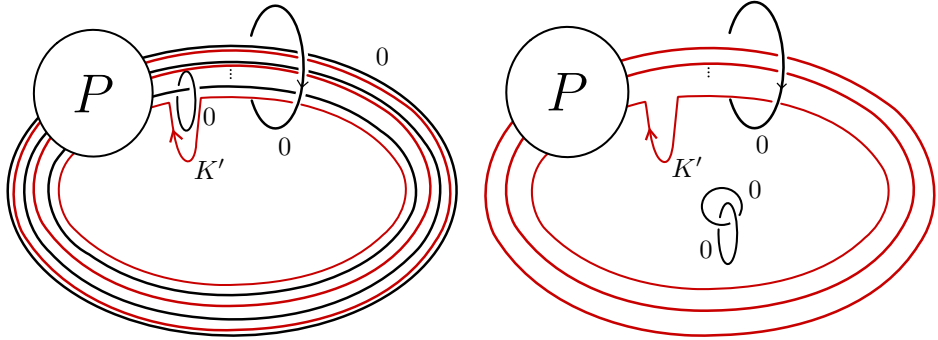
\begin{figure}[htb]
\begin{tikzpicture}
\node at (0,0) {\includegraphics[width=6cm]{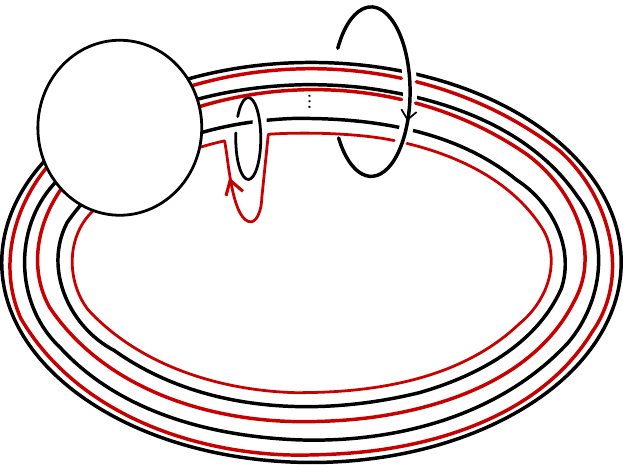}};
\node at (-1.85,1) {\Huge $P$};
\node at (2,1.5) {$0$};
\node at (0.7,0.3) {$0$};
\node at (-0.3,0.7) {$0$};
\node at (-0.3,0.0) {$K'$};
\end{tikzpicture}
\begin{tikzpicture}
\node at (0,0) {\includegraphics[width=6cm]{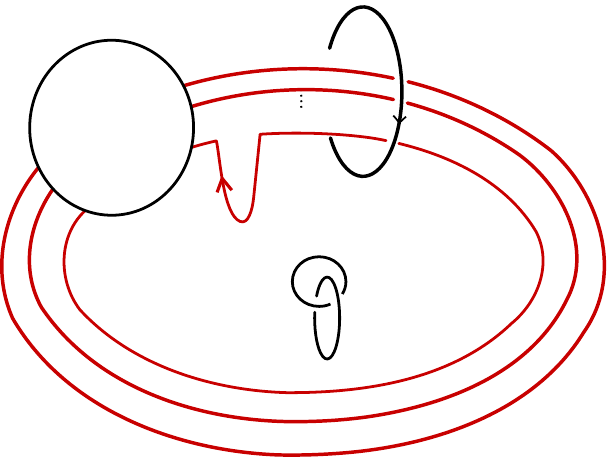}};
\node at (-1.9, 1) {\Huge $P$};
\node at (-0.05, -1.1) {$0$};
\node at (0.7,0.3) {$0$};
\node at (0.6, -0.3) {$0$};
\node at (-0.3,0.0) {$K'$};
\end{tikzpicture}
\caption{Left: An isotopy of the knot $K'$ from Figure~\ref{fig:diffeomorphic-knot-1}. Note that a helper circle has been produced. Right: The result of several handleslides over the helper circle.}\label{fig:diffeomorphic-knot-2}
\end{figure}
Slide $K'$ over the $2$-handle tied into the tangle $P$
resulting in the diagram on the left of Figure~\ref{fig:diffeomorphic-knot-2}. This is just an isotopy of the knot in $M$, and it produces a helper circle. Use this helper circle to separate the $2$-handle we just slid over from the rest of the diagram, producing the diagram on the right of Figure~\ref{fig:diffeomorphic-knot-2}.
The resulting Hopf link with $0$-framed components simply indicates a connected-sum with $S^3$, revealing that there exists a diffeomorphism from $M$ to $S^1 \times S^2$ which sends $K'$ to $K$. This diffeomorphism is orientation preserving since each constituent step is orientation preserving.\end{proof}

\begin{const}\label{const:bordism}
\begin{figure}[htb]
\begin{tikzpicture}
\node at (0,0) {\includegraphics[width=6cm]{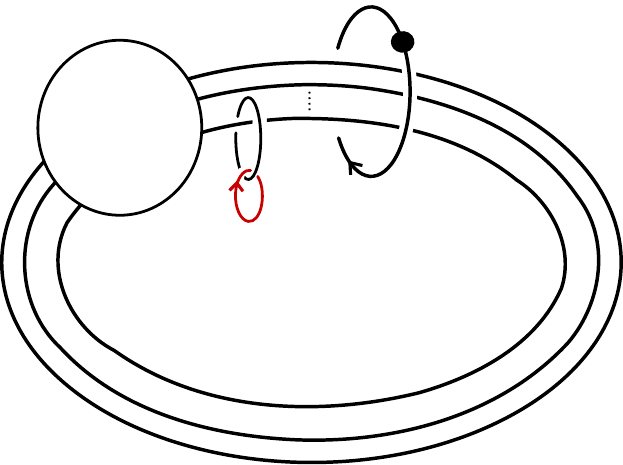}};
\node at (-1.85,1) {\Huge $P$};
\node at (2,1.5) {$0$};
\node at (-0.2,0.8) {$\langle 0\rangle$};
\node at (-0.3,0.0) {$H$};
\end{tikzpicture}
\caption{A cobordism $W$ from $S^1\times S^2$ to $M\cong S^1\times S^2$. The cylinder $H\times I\subseteq W$ is cobounded by the Hopf knot in $\bdry^-W=S^1\times S^2$ and the knot $K'\subseteq\bdry^+W=M$.}\label{fig:P-handle-decomp}
\end{figure}
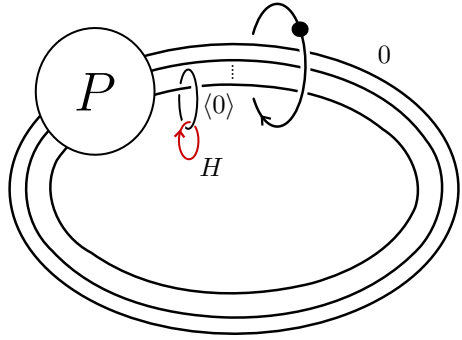
Consider the relative Kirby diagram in Figure~\ref{fig:P-handle-decomp} and its realization $W$. Its negative boundary is the realization of Figure~\ref{fig:Hopf-knot} and thus is identified with $S^1 \times S^2$. The positive boundary $\partial^+ W$ is the manifold~$M$, since we obtain a diagram for $\partial^+W$ by simply replacing the unknotted circle with a dot on it by a parallel $0$-framed unknot; this results in Figure~\ref{fig:diffeomorphic-knot-1}. Note that the annulus $H\times I \subset W$ is cobounded by the Hopf knot $H$ and the knot $K'\subset M$.
\end{const}

\begin{proposition}\label{prop:concordance-to-modified-knot}
Let $w$ be the winding number of the knot $K$, where $w\neq 0$. Let $R$ be any ring in which $w$ is a unit. The cobordism~$(W; \psi, \overline\Id)$, where $\psi$ is the orientation preserving diffeomorphism from Lemma~\ref{lem:modify-knot} is an $R$-homology cobordism.
\end{proposition}
\begin{proof}
To check that $W$ is a homology cobordism, consider the
relative chain complex over $R$ for $(W, \bdry^- W)$:
\[
\xymatrix{
C_3 = 0 \ar[r] &  C_2= R \ar[r]^{\times w} &  C_1 = R \ar[r] & C_0 = 0
}. \]
Since $w$ is a unit in $R$, we see that $H_*(W, \bdry^- W;R)=0$ and
by Poincar\'e duality $H_*(W, \bdry^+ W;R)=0$.
This implies that both maps
\begin{align*}
H_*(\bdry^- W;R) &\to H_*(W;R)\\
H_*(\bdry^+ W;R) &\to H_*(W;R)
\end{align*}
are isomorphisms and so $W$ is a homology cobordism.
\end{proof}

Note that $K$ and $H$ are $R$-concordant in the manifold $W$ above, via the annulus $H\times I$. Thus, this completes the proof of Theorem~\ref{thm:thmAprime}. Indeed, we have proved the following slightly more general result.

\begin{corollary}\label{cor:homology-concordant-to-Hopf}
Let $K$ be any knot in $S^1\times S^2$ with winding number $w\neq 0$, and $R$ a ring in which $w$ is a unit.  Then $K$ is $R$-concordant to the Hopf knot.
\end{corollary}

In the remainder of this section, we focus on knots with winding number~$1$. In this case, the cobordism~$W$ is in fact a product, as follows.
\begin{lemma}\label{lem:handle-decomp}
Let $K\subset S^1\times S^2$ be a winding number $1$ knot and
$(W; \psi, \overline\Id)$ the corresponding cobordism given by Construction~\ref{const:bordism}. Then there is a orientation preserving diffeomorphism of pairs
\[(\phi,\Id)\colon (S^1 \times S^2 \times I, S^1 \times S^2 \times \{0\})\rightarrow (W, \partial^- W)\]
i.e.\ the map $\phi$ restricts to the identity map on $S^1\times S^2\times \{0\}=\partial^-W$.
\end{lemma}

\begin{proof}

We show that $W$ is diffeomorphic to $S^1\times S^2\times I$ relative to the identification of its negative boundary. We will do this in Figure~\ref{fig:P-handle-decomp} by sliding the $2$-handle attaching curves over the handles of $\partial^-W$. In particular, the key point is that the $\langle0\rangle$-framed surgery curve in $\partial^- W$ can serve as a helper circle and affect crossing changes in $K$. Suppose that the geometric winding number of $K$ about the $1$-handle is strictly greater than the (algebraic) winding number, which is 1. Then there must be a \emph{returning arc}, i.e.  an arc that starts near the dotted circle, enters the circle marked $P$, and turns around. Using the helper circle, unlink the returning arc from the rest of $P$ by crossing changes. Then the returning arc can be isotoped out of the tangle $P$, and off the $1$-handle, decreasing the geometric winding number by two. Iterate this process until the geometric winding number equals the winding number.

Now the $2$-handle has geometric winding number one around the $1$-handle, and thus these handles cancel each other. This leaves a single unknotted circle decorated with $\langle 0\rangle$ and so $(W,\bdry^- W)$ is diffeomorphic to $(S^1\times S^2\times I, S^1\times S^2\times \{0\})$.
\end{proof}

\noindent We now have all the ingredients for a proof of Theorem~\ref{thm:thmA}.

\begin{proof}[Proof of Theorem~\ref{thm:thmA}]
Consider the cobordism~$W$ from Construction~\ref{const:bordism}. Note that we wish to build a concordance from a given winding number $1$ knot $K$ to $H$; thus, it is not sufficient to work in the cobordism $W$ even though it is diffeomorphic to $S^1\times S^2\times I$ -- we must work in $S^1\times S^2\times I$ itself. In $W$ we have a cylinder cobounded by $K'$ and $H$, and from Lemma~\ref{lem:modify-knot} we have an orientation preserving diffeomorphism $\psi\colon \partial^+W=M\rightarrow S^1\times S^2$ such that $\psi(K')=K$. Let $g$ denote the composition $\psi \circ \phi\big|_{S^1\times S^2\times \{1\}}$, where $\phi$ is the orientation preserving diffeomorphism from Lemma~\ref{lem:handle-decomp}. Then $g\colon S^1\times S^2\rightarrow S^1\times S^2$ is an orientation preserving diffeomorphism, and in $S^1\times S^2\times I$ we have a cylinder cobounded by $g^{-1}(K)$ and $H$. Apply $g\times \Id_{I}$ to $S^1\times S^2\times I$. This yields a cylinder in $S^1\times S^2\times I$ cobounded by $K$ and $g(H)$. We will now show that $g(H)$ is isotopic to the Hopf knot, which will complete the proof.

First we note that $g$ induces the identity homomorphism on $\pi_1(S^1\times S^2)$, as follows. The class $[H]$ is a generator of $\pi_1(S^1 \times S^2)$. Since $H$ and $g^{-1}(K)$ cobound an annulus in $S^1\times S^2\times I$, $[g^{-1}(K)]=[H]$ in homotopy, i.e.\ $g([H])=[K]$. Since $K$ is assumed to be winding number $1$, $[K]=[H]$, as needed. Now we invoke a result of Gluck~\cite{Gluck62}, which states that the mapping class group $\pi_0(\Diff(S^1 \times S^2))$ of $S^1\times S^2$ is isomorphic to $\Z/2\oplus \Z/2\oplus \Z/2$, where the first summand is generated by the diffeomorphism which reverses the orientation on $S^1$; the second is generated by the diffeomorphism which reverses the orientation on $S^2$; and the third is generated by the Gluck twist
\begin{align*}
\gamma \colon S^1 \times S^2 &\to S^1 \times S^2\\
(t, z) &\mapsto (t, p_t(z)),
\end{align*}
where $p_t \in \pi_1(\SO(3)) \cong \Z/2$ is the non-trivial element. Since $g$ preserves orientation and acts by the identity on $\pi_1(S^1\times S^2)$, it must lie in the third $\Z/2$ summand, i.e.\  it is isotopic either to the identity map or to the Gluck twist. Either of these diffeomorphisms preserve the isotopy class of the Hopf knot.
\end{proof}

We further extend our results to knots in $3$-manifolds that bound homology $S^1\times D^3$'s.
\begin{proposition}\label{prop:knots-in-homology-S1xS2}
Let $w$ be a non-zero integer and $R\supset \Z$ be a ring extension in which $w$ is a unit. Let $M$ be a closed oriented $3$-manifold with $H_1(M;R) \cong R$. Suppose that $M$ bounds a compact oriented $4$-manifold~$W$
such that $H_1(M;R)\to H_1(W;R)$ is an isomorphism and $H_2(W;R)=0$.
Let $K$ be a knot in $M$ representing $w$ times the generator of $H_1(M)$. Then $K$ is $R$-cobordant to the Hopf knot.
\end{proposition}

We note that while the requirement that $M$ bound such a $4$-manifold might seem excessive at first glance, without this assumption there would exist no $R$-homology cobordism from $M$ to $S^1\times S^2$.

\begin{proof}
We build the cobordism as follows. Perform a boundary connected sum of $S^1\times S^2\times I$ with $W$ within $S^1\times S^2\times \{1\}$. Next,
add a $2$-handle to the new boundary component $S^1\times S^2\# M$ along a curve $P$ which first follows $K$ and then the Hopf knot, with any framing. On the level of $3$-manifolds the $2$-handle cancels with the added $S^1\times S^2$-factor (perform a slam dunk move); thus, we have built a cobordism from $S^1\times S^2$ to $M$. We can check that this is an $R$-cobordism -- this is very similar to the proof of Proposition~\ref{prop:concordance-to-modified-knot}. In this
cobordism, the cylinder $H\times I$ is cobounded by $H\times\{0\}$ and $H\times \{1\}$, and sliding over the curve $P$ shows that $H\times\{1\}$ is isotopic to $K$.
\end{proof}


\section{Slice knots in $S^1\times S^2$}\label{sec:slice-knots}
Recall that a knot~$K$ in $S^1\times S^2$ is slice (resp.\ topologically slice) if it bounds a smoothly (resp. locally flatly) embedded disk in $D^2\times S^2$. In this section, we prove Theorem \ref{thm:thmB}, that is, we show that there are infinitely many topological concordance classes of slice knots in $S^1\times S^2$ for any fixed winding number $w\neq \pm1$.

\begin{remark}
The reader may note that we could instead have considered knots which bound disks in $S^1\times D^3$ instead. Recall that the standard diagram for $S^1\times S^2$ consists of an unknot $U$ in $S^3$ decorated with a $0$, and a Kirby diagram for $S^1\times D^3$ is obtained by simply removing the $0$ and putting a dot on $U$. Then the question of whether a knot $K$ in $S^1\times S^2$ bounds a disk in $S^1\times D^3$ is really a question about whether the link $K\sqcup U$ is slice in $D^4$, where the slice disk for $U$ is the standard unknotted slice disk, i.e.\ isotopic to $D^2\times\{0\}\subseteq D^2\times D^2\cong D^4$.  Moreover, in order for a knot to be slice in $S^1\times D^3$ it should first be null homotopic, and thus, have winding number zero. We focus on sliceness in $D^2\times S^2$ since it gives us the opportunity to explore different winding numbers.
\end{remark}

\begin{proof}[Proof of Theorem~\ref{thm:thmB}]
For $w=0$, we follow the proof of~\cite[Corollary 8.4 (2)]{FNOP16}, using covering links. We recall their example now. Let $\{J_i\}_{i\geq 1}$ be an infinite family of knots in~$S^3$ such that when $i\neq j$ the knots~$J_i\#rJ_i$ and $J_j\#rJ_j$ are not topologically concordant. For example, we may take $J_i$ to be the connected sum of $i$ copies of the right handed trefoil, in which case the signature detects that $J_i\#rJ_i$ is not topologically concordant to $J_j\#rJ_j$. Consider the winding number zero knot~$K_0(J_i)$ depicted in Figure~\ref{fig:winding-zero-slice-knots}. The depicted band move reduces~$K_0(J_i)$ to two parallel Hopf knots with opposite orientations. Thus $K_0(J_i)$ is slice.
\begin{figure}[htb]
\begin{tikzpicture}
\node at (-3,0) {\includegraphics[width=3.5cm]{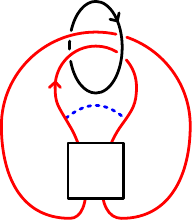}};
\node at (-3.5, 2) {$0$};
\node at (-3, -1.1) {$J$};
\node at (0.7, -1.1) {};
\end{tikzpicture}
\begin{tikzpicture}
\node at (0,0) {\includegraphics[width=3.5cm]{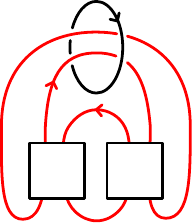}};
\node at (-0.5, 2) {$0$};
\node at (-0.7, -1.1) {$J$};
\node at (0.7, -1.1) {$J$};
\end{tikzpicture}
\caption{Left: For any knot $J$, the depicted band move shows that~$K_0(J)$ is slice.
Right: The $2$-fold covering link~$\widetilde{K_0(J)}$
of $K_0(J)$ is a $2$-component link, each component of which is the distant knot~$\Phi(J\#rJ)$.}\label{fig:winding-zero-slice-knots}
\end{figure}
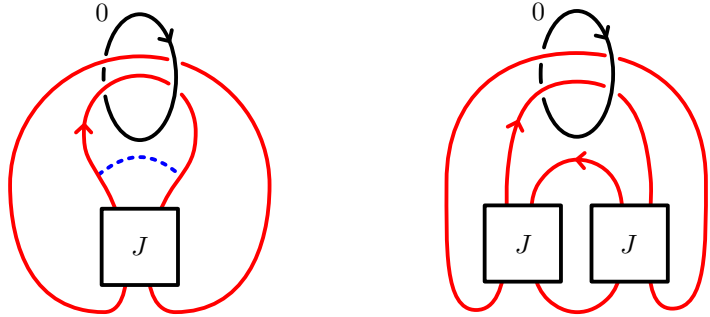

Suppose for the sake of contradiction that there exists a locally flat concordance $C$ in $S^1\times S^2\times I$ between $K_0(J_i)$ to $K_0(J_j)$, for some $i\neq j$. Lift the knots $K_0(J_i)$ and $K_0(J_j)$ as well as the concordance $C$ to the double cover of $S^1\times S^2\times I$. Conveniently, the double cover of $S^1\times S^2$ is again $S^1\times S^2$. Since $K_0(J_i)$ is winding number zero, it lifts to
a $2$-component link~$\widetilde{K_0(J_i)} = \widetilde{K^1_0(J_i)} \sqcup \widetilde{K^2_0(J_i)}$, shown in Figure~\ref{fig:winding-zero-slice-knots}. Similarly, the concordance~$C$ lifts to two disjoint
annuli~$\widetilde{C} = \widetilde{C_1} \sqcup \widetilde{C_2}$,
where each $\widetilde{C_k}$ is a locally flat concordance between $\widetilde{K^k_0(J_i)}$ and $\widetilde{K^k_0(J_j)}$. Notice that each component~$\widetilde K^k_0(J_i) = \Phi(J_i\#rJ_i)$ is a distant knot, and thus by Theorem~\ref{thm:distant-knots-cobordisms}, it follows that $J_i\#rJ_i$ is topologically concordant to $J_j\#rJ_j$ in $S^3 \times I$ contradicting our assumption to the contrary.

For the remainder of the proof it is sufficient to construct examples with positive winding numbers, since we may obtain examples for negative winding numbers by simply changing orientations. Our examples will all be of the form $K_w(J)$, indicated in the left panel of Figure~\ref{fig:many-slice-knots}, for some choice of knots $J$ which we will describe soon. We see that these knots are all slice, for any choice of knots $J$, via the band moves shown in Figure~\ref{fig:many-slice-knots}.
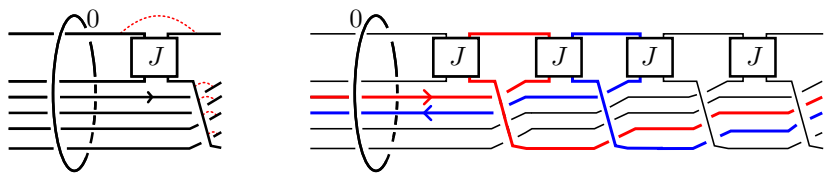
\begin{figure}[htb]
\begin{tikzpicture}
\node at (0,0) {\includegraphics[height=2.1cm]{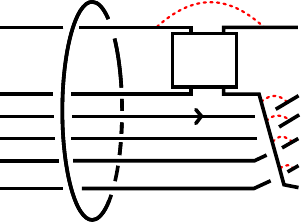}};
\node at (-0.3, 1) {$0$};
\node at (0.54, 0.5) {$J$};
\node at (6,0) {\includegraphics[height=2.1cm]{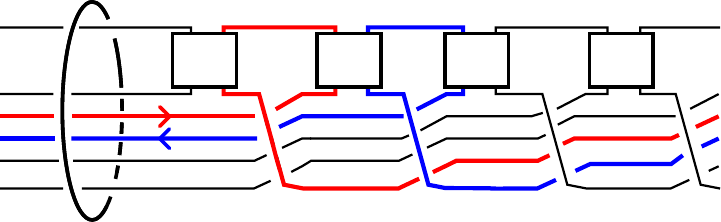}};
\node at (-0.3 + 3.5, 1) {$0$};
\node at (4.55, 0.5) {$J$};
\node at (5.9, 0.5) {$J$};
\node at (8.5, 0.5) {$J$};
\node at (7.1, 0.5) {$J$};
\end{tikzpicture}
\caption{Left: The knot $K_4(J)$ in $S^1\times S^2$ is shown, with winding number $4$. The indicated band moves show that $K_4(J)$ is slice. Right:  The $4$-fold cover
of $K_4(J)$, $\widetilde{K_4(J)}$, is a $4$-component link.
Two components of this link are highlighted.}\label{fig:many-slice-knots}
\end{figure}

First, let $w=2$. Let $\{J_i\}_{i\geq 1}$ be an infinite family of knots in $S^3$ such that if $i\neq j$ the knots~$(J_i \# rJ_i)_{2,1}$ and $(J_j \# rJ_j)_{2,1}$ do not cobound a locally flat orientable surface in $S^3 \times I$ with genus three (the latter knots are $(2,1)$ cables, with longitudinal winding two, similar to our notation from Figure~\ref{fig:Hopf-knot-cable}.) For example, we may take $J_i$ to be the connected sum of $2i$ copies of the right handed trefoil, in which case the Levine-Tristram signature function detects that $(J_i \# rJ_i)_{2,1}$ and $(J_j \# rJ_j)_{2,1}$ do not cobound a locally flat genus three orientable surface in $S^3 \times I$. Suppose for the sake of contradiction that there exists a locally flat concordance~$C$ between $K_2(J_i)$ and $K_2(J_j)$ in $S^1\times S^2\times I$ for some $i\neq j$. Let $\widetilde{K_2(J_i)}$ and $\widetilde{K_2(J_j)}$ be the preimages of $K_2(J_i)$ and $K_2(J_j)$ respectively in the $2$-fold cover of $S^1\times S^2$. Orient the two components of the covering links and perform band sums as shown in Figures~\ref{fig:2-many-slice-knots} and~\ref{fig:2-many-slice-knots-2}. This shows that there exist locally flat orientable genus one cobordisms from $\widetilde{K_2(J_i)}$ to $\Phi((J_i \# rJ_i)_{2,1})$ and from $\widetilde{K_2(J_j)}$ to $\Phi((J_j \# rJ_j)_{2,1})$. Glue these onto $C$ to obtain a genus three cobordism between $\Phi((J_i \# rJ_i)_{2,1})$ and $\Phi((J_j \# rJ_j)_{2,1})$. By Theorem~\ref{thm:distant-knots-cobordisms}, this implies that $(J_i \# rJ_i)_{2,1}$ and  $(J_j \# rJ_j)_{2,1}$ cobound a locally flat genus three surface in $S^3\times I$, which is a contradiction.
\begin{figure}[htb]
\begin{tikzpicture}
\node at (0,0) {\includegraphics[height=2.1cm]{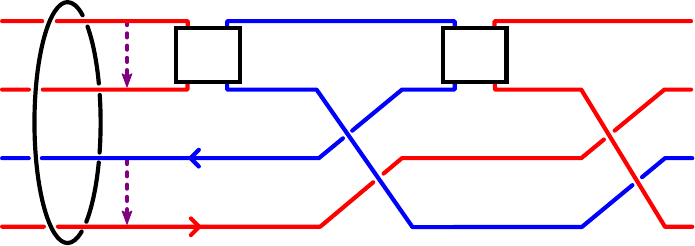}};
\node at (-1.2, 0.6) {$J$};
\node at (1.1, 0.6) {$J$};
\node at (-2.1, 1.1) {$0$};
\node at (6.5,0) {\includegraphics[height=2.1cm]{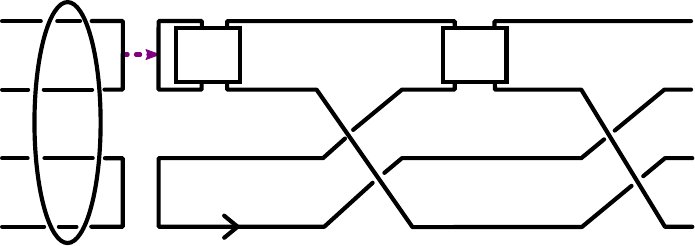}};
\node at (4.4, 1.1) {$0$};

\node at (5.32, 0.6) {$J$};
\node at (7.6, 0.6) {$J$};
\end{tikzpicture}
\caption{Left: The $2$-fold cover
of $K_2(J)$, $\widetilde{K_2(J)}$ is a $2$-component link. Right: The result of a band move from the left.}\label{fig:2-many-slice-knots}
\end{figure}
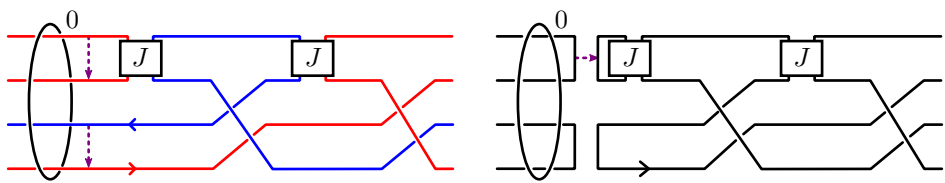
\begin{figure}[htb]
\begin{tikzpicture}
\node at (0,0) {\includegraphics[width=5.5cm]{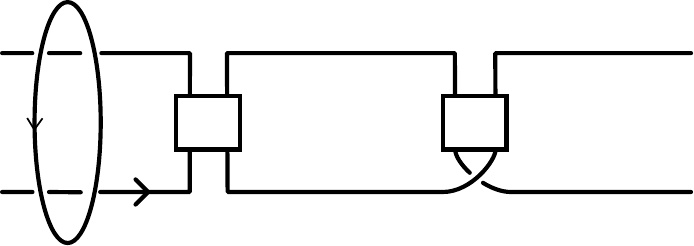}};
\node at (-2, 1) {$0$};
\node at (1.02, 0.01) {$J$};
\node at (-1.1, 0.01) {$J$};
\node at (6,0) {\includegraphics[width=5.5cm]{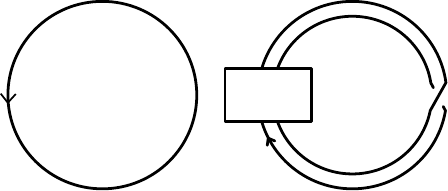}};
\node at (5.4, 1) {$0$};
\node at (6.53, 0) {$J\#rJ$};
\end{tikzpicture}
\caption{Left: The result of another band sum from the right panel of Figure~\ref{fig:2-many-slice-knots}.
Right: An isotopy reduces this knot to a distant knot.}
\label{fig:2-many-slice-knots-2}
\end{figure}
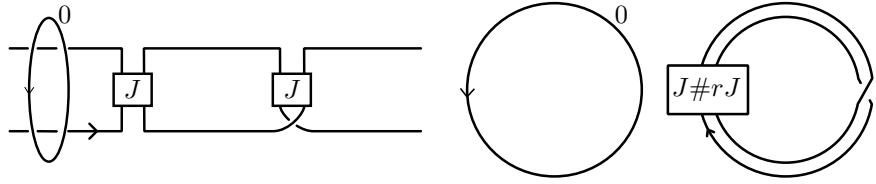

It remains to address the case of $w>2$. Let $\{J_i\}$ be an infinite family of knots in $S^3$ such that if $i\neq j$ the knots~$J_i\#J_i\#(rJ_i)_{2,1}$ and $J_j\#J_j\#(rJ_j)_{2,1}$ do not cobound a locally flat genus one orientable surface in $S^3 \times I$. Again, one might take $J_i$ to be the connected sum of $i$ copies of the right handed trefoil, in which case the signature detects that $J_i\#J_i\#(rJ_i)_{2,1}$ and $J_j\#J_j\#(rJ_j)_{2,1}$ do not cobound a locally flat genus one orientable surface in $S^3 \times I$.
\begin{figure}[htb]
\begin{tikzpicture}
\node at (0,0) {\includegraphics[width=5.5cm]{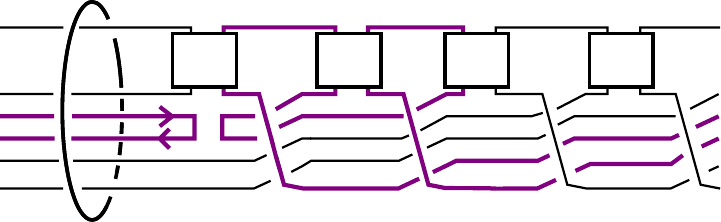}};
\node at (-1.9, 0.8) {$0$};

\node at (-1.2, 0.4) {$J$};
\node at (-.1, 0.4) {$J$};
\node at (.9, 0.4) {$J$};
\node at (2, 0.4) {$J$};

\node at (6,0) {\includegraphics[width=5.5cm]{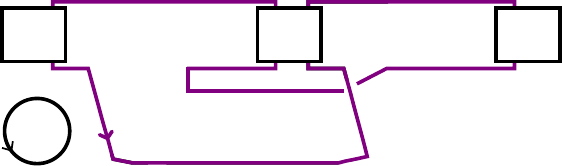}};
\node at (3.8, -0.04) {$0$};
\node at (3.6, 0.46) {$J$};
\node at (6.1, 0.46) {$J$};
\node at (8.4, 0.46) {$J$};
\end{tikzpicture}
\caption{Left: Band summing two components of $\widetilde{K_4(J)}$.
Right: an isotopy reduces this knot to a distant knot.}
\label{fig:many-slice-knots-2}
\end{figure}
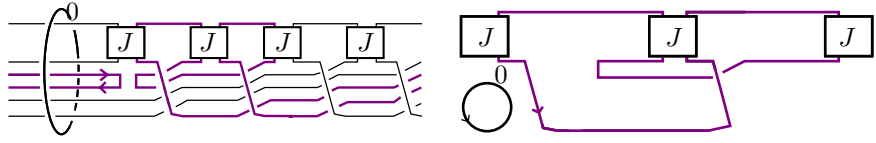
Suppose for the sake of contradiction that there exists a locally flat concordance~$C$ between $K_w(J_i)$ and $K_w(J_j)$ in $S^1\times S^2\times I$. Since $K_w(J_i)$ is winding number~$w$, the lift~$\widetilde{K_w(J_i)}$ to the $w$-fold cover of $S^1\times S^2$  has $w$ components (see Figure~\ref{fig:many-slice-knots}).
By lifting the concordance~$C$ to the $w$-fold cover,
we see $w$ disjoint locally flat concordances between the components of~$\widetilde{K_w(J_i)}$ and the components of $\widetilde{K_w(J_j)}$.

Now, if we band together two components of $\widetilde{K_w(J_i)}$ as shown in Figure~\ref{fig:many-slice-knots-2} and do the same with two components of $\widetilde{K_w(J_j)}$ and ignore all other components, we see that the resulting knots cobound a locally flat genus one orientable surface.
Indeed the resulting knots are the distant knots $\Phi(J_i\#J_i\#(rJ_i)_{2,1})$ and
$\Phi(J_j\#J_j\#(rJ_j)_{2,1})$ respectively. By Theorem~\ref{thm:distant-knots-cobordisms}, $J_i\#J_i\#(rJ_i)_{2,1}$
and $J_j\#J_j\#(rJ_j)_{2,1}$ also cobound a locally flat genus one orientable surface in $S^3\times I$, which is a contradiction. This completes the proof.
\end{proof}


\section{Nonslice knots with even winding number.}\label{sec:nonslice-even}

This section gives the proof of Theorem~\ref{thm:thmC} for knots with even winding number. We address knots with odd winding numbers in the next section with quite different techniques.

We first recall a result of Murakami and Yasuhara~\cite{MY00} giving a lower bound on $\gamma^{\top}_4(J)$ for knots $J$ in $S^3$ (cf.~\cite{GL11}). For a given knot $J$, let $\Sigma(J)$ be the double branched cover of $S^3$ branched along $J$ and let $(H_1(\Sigma(J)),\lk)$ denote the linking form on $\Sigma(J)$.

\begin{corollary}[{\cite[Corollary~$2.6$]{MY00}}] \label{corollary:MY} Suppose that $H_1(\Sigma(J)) = \mathbb{Z}_n$ where $n$ is a product of primes with odd exponent. If $\gamma^{\top}_4(J)=1$, there is a generator $a \in H_1(\Sigma(J))$ such that $\lk(a,a) = \pm\frac{1}{n}$.
\end{corollary}

As an application of the above, Murakami and Yasuhara show that the figure-eight knot~$4_1$ does not bound a M\"{o}bius band in $D^4$. Note that the obstruction only depends on the linking form of the double branched cover. Therefore, we can immediately find an infinite family of pairwise topologically non-concordant knots in $S^3$ that do not bound M\"{o}bius bands in $B^4$ by simply connect summing $i$ copies of a knot with determinant one and non-trivial signature to the figure-eight knot (recall that the determinant of a knot is the order of the first homology group of its double branched cover).

\begin{theorem}\label{thm:non-slice-even}
For any $w \in 2\mathbb{Z}$, there exists an infinite family of winding number $w$ knots, none of which is topologically slice, and which are pairwise distinct up to topological concordance.
\end{theorem}
\begin{proof} When $w = 0$, let $\{J_i\}_{i\geq 1}$ be an infinite family of pairwise topologically non-concordant knots in $S^3$ that are not topologically slice. Then the corresponding distant knots $\{\Phi(J_i)\}$ satisfy the desired properties by Theorem~\ref{thm:distant-knots-cobordisms} and Corollary~\ref{cor:slice-distant-knots}.

For $w>0$, let $\{J_i\}_{i\geq 1}$ be an infinite family of knots in $S^3$ such that $\gamma^{\top}_4(J_i) > 1$ and $J_i \#rJ_i$ and $J_j \# rJ_j$ do not cobound a genus $1$ locally flat orientable surface in $S^3 \times I$. As we noted above, this can be achieved by taking the connect sum of $i$ copies of a knot with determinant one and signature at least $2$ (e.g. $10_{124}$) with the figure-eight knot.

\begin{figure}[htb]
\begin{tikzpicture}
\node at (0,0) {\includegraphics[width=5.5cm]{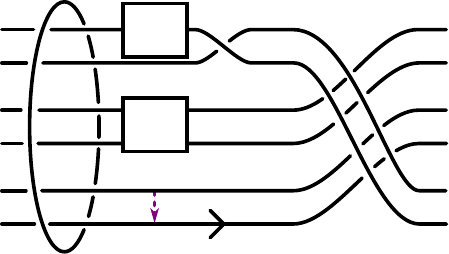}};
\node at (-1.6, 1.4) {$0$};
\node at (-0.85, 1.2) {$J$};
\node at (-0.85, 0.01) {$\frac{2-w}{2}$};
\node at (6,0) {\includegraphics[width=5.5cm]{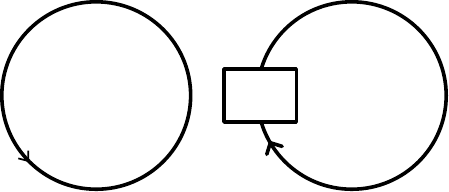}};
\node at (5.4, 1) {$0$};
\node at (6.44, 0) {$J$};
\end{tikzpicture}
\caption{Left: A winding number $w \in 2\mathbb{Z}$ knot $K_w(J)$ in $S^1\times S^2$. The specific case $w=6$ is shown. The box containing $\frac{2-w}{2}$ indicates $\frac{2-w}{2}$ full right-handed twists.  Right: The distant knot $\Phi(J)$ obtained from $K_w(J)$ by the non-orientable band sum shown on the left.}
\label{fig:nonsliceevenslice}
\end{figure}
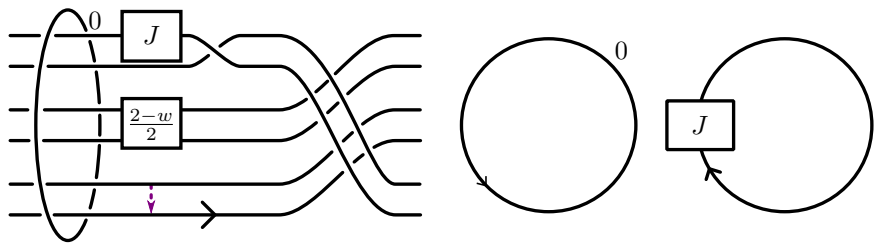
Consider the knot $K_w(J_i)$, with even winding number $w>0$, described in Figure~\ref{fig:nonsliceevenslice}. The suggested non-orientable band move on $K_w(J_i)$ yields the distant knot $\Phi(J_i)$. If we assume that $K_w(J_i)$ is topologically slice then $\gamma^{\top}_4(\Phi(J_i)) \leq 1$. Thus, by Corollary~\ref{cor:slice-distant-knots}, $\gamma^{\top}_4(J_i) \leq 1$, which is a contradiction. Hence $K_w(J_i)$ is not topologically slice.

Now we prove that $K_w(J_i)$ is not topologically concordant to $K_w(J_j)$, if $i \neq j$. This will be similar to our proof of Theorem~\ref{thm:thmB}. Suppose we have a topological concordance $C$ between the two knots. The preimage of $C$ in the $w$-fold cover of $S^1\times S^2\times I$ is a disjoint union of topological concordances between the components of the $w$-component links $\widetilde{K_w(J_i)}$ and $\widetilde{K_w(J_j)}$, the lifts of $K_w(J_i)$ and $K_w(J_j)$.
\begin{figure}[htb]
\begin{tikzpicture}
\node at (-3.0,0) {\includegraphics[width=10cm]{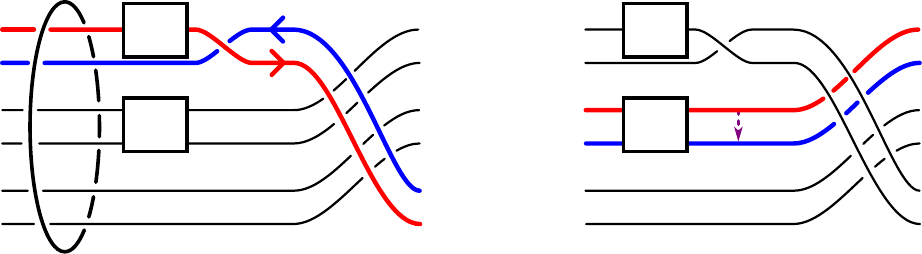}};
\node at (-2.5, 0.9) {$\cdots$};
\node at (-2.5, 0) {$\cdots$};
\node at (-2.5, -0.9) {$\cdots$};
\node at (-0.88, 1.05) {$J$};
\node at (-0.89, 0.01) {$\frac{2-w}{2}$};
\node at (-6.31, 1.05) {$J$};
\node at (-6.32, 0.01) {$\frac{2-w}{2}$};
\node at (-7, 1.3) {$0$};
\end{tikzpicture}
\caption{The $w$-component link $\widetilde{K_w(J)}$ in $S^1\times S^2$ and a band move that transforms a $2$-component sublink into the distant knot $\Phi(J\#rJ)$.}
\label{fig:nonsliceevendistinct}
\end{figure}
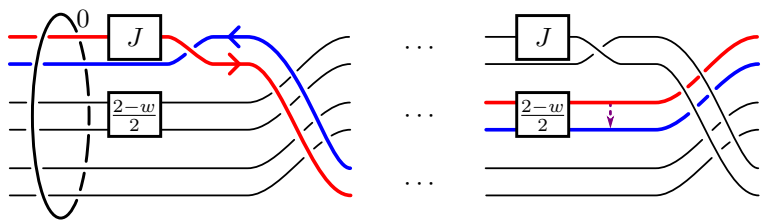
Band two adjacent components of $\widetilde{K_w(J_i)}$ as shown in Figure~\ref{fig:nonsliceevendistinct} to get a cobordism to the distant knot $\Phi(J_i\# rJ_i)$. Doing the same to $\widetilde{K_w(J_j)}$ yields a locally flat genus one cobordism between the distant knots $\Phi(J_i\#rJ_i)$ and $\Phi(J_j\#rJ_j)$. By Theorem~\ref{thm:distant-knots-cobordisms} we conclude that $J_i\#rJ_i$ and $J_i\#rJ_i$ cobound a locally flat genus one surface in $S^3 \times I$, which is a contradiction.

For negative even winding numbers, reverse the orientations of the positive winding number examples.
\end{proof}

Note that our technique above relies on finding a cobordism from a given knot to a distant knot. If the given knot has non-zero winding number, it is impossible to find an orientable cobordism to a distant knot, since orientable cobordisms preserve winding number. On the other hand, non-orientable cobordisms preserve the winding number mod $2$, and thus, it is impossible to find a cobordism from a knot with odd winding number to a distant knot. As a result, we need a completely different approach for finding non-slice knots with odd winding number. We do this in the following section.

Additionally, we constructed a family of knots $K_w(J_i)$ such that each $K_w(J_i)$ cobounds a non-orientable surface with non-orientable genus one with the distant knot $\Phi(J_i)$. Therefore, a family of knots $\{J_i\}$ in $S^3$ with arbitrarily large $\gamma_4^{top}$ would yield the family of knots $\{K_w(J_i)\}$ in $S^1\times S^2$ with arbitrarily large $\gamma_4^{top}$, for any even winding number $w$. However, such a family of knots in $S^3$ has not been constructed as far as we know.

However, there exist knots with arbitrarily large smooth non-orientable $4$-genus~\cite{Bat14, OSS15}. Furthermore, it is possible to choose these examples to be topologically slice~\cite{FPR16}. This yields the following families of knots.

\begin{proposition}\label{prop:evencrosscap} For any $w \in 2\mathbb{Z}$, there exists an infinite family of topologically slice knots $\{K_i\}_{i\geq 1}$ in $S^1\times S^2$ with winding number $w$ such that  $2g_4(K_i) \geq \gamma_4(K_i) \geq i$.
\end{proposition}
\begin{proof} When $w = 0$, let $\{J_i\}$ be an infinite family of topologically slice knots in $S^3$ with $\gamma_4(J_i)\geq i$ from~\cite{FPR16}; for example, let $J_i$ be the connected sum of $i$ copies of the positive clasped untwisted Whitehead double of the right handed trefoil. Then $\gamma_4(\Phi(J_i))=\gamma_4(J_i)\geq i$ by Corollary~\ref{cor:slice-distant-knots}, as needed.

For $w>0$, let $\{J_i\}$ be an infinite family of topologically slice knots in $S^3$ with $\gamma_4(J_i) \geq i+1$; e.g.\ take the above family from~\cite{FPR16} with a shifted index. Consider the knots $K_w(J_i)$ with winding number $w$ described in Figure~\ref{fig:nonsliceevenslice}. Since $J_i$ is topologically slice,  $K_w(J_i)$ is topologically concordant to the $(2,1)$-cable of $H_{w,1}$ which is topologically slice. Thus, the knot $K_w(J_i)$ is topologically slice. On the other hand, the knot $K_w(J_i)$ and the distant knot $\Phi(J_i)$ cobound a M\"obius band (see Figure~\ref{fig:nonsliceevenslice}). Therefore by applying Corollary~\ref{cor:slice-distant-knots} we see that $\gamma_4(K_w(J_i)) \geq \gamma_4(\Phi(J_i)) -1 \geq  i$.

Once again, when $w<0$ simply take reverses of the examples for positive winding numbers.
\end{proof}

Interestingly, the above phenomenon for non-orientable genus is impossible for knots in $S^1\times S^2$ with odd winding number.

\begin{proposition}\label{prop:oddcrosscap}
Let $K$ be a knot in $S^1\times S^2$ with odd winding number~$w=2k+1$.
Then $\gamma^{top}_4(K) \leq \gamma_4(K)\leq k$.
\end{proposition}
\begin{proof}
By adding $k$ non-orientable bands, we obtain a smooth cobordism $F$ from the knot $K$ to a knot $K'$ with winding number $1$. By Theorem~\ref{thm:thmA}, any winding number $1$ knot is smoothly slice. Gluing together $F$ and a smooth slice disk for $K'$ yields a smooth non-orientable surface of genus~$k$ bounded by $K$.
\end{proof}

\section{Nonslice knots in odd winding number.}\label{sec:nonslice-odd}

In the previous section we proved that, for any even winding number, there exists an infinite family of knots in $S^1\times S^2$ which are not topologically slice in $D^2\times S^2$. This section will address the case of odd winding numbers and complete the proof of Theorem~\ref{thm:thmC}. Our proof will consist of a detailed analysis of a sliceness obstruction due
to Gilmer and Livingston~\cite{Gilmer83} in terms of Casson-Gordon invariants. Recall that Theorem~\ref{thm:thmC} states that for any winding number $w$ other than $\pm 1$ and $\pm 3$, there exists an infinite family of winding number $w$ knots which are distinct in topological concordance, and such that none of them is topologically slice. By Theorem~\ref{thm:thmA} we know that winding number $\pm 1$ knots are smoothly (and thus, topologically) slice. As we noted in the introduction, we believe that the exclusion of winding number $\pm 3$ knots from Theorem~\ref{thm:thmC} is a limitation of our technique, rather than being indicative of a deeper phenomenon.

The bulk of this section will consist of the proof of the following proposition. We show first how it yields a proof of Theorem~\ref{thm:thmC}.

\begin{proposition}\label{prop:nonslice}
Let $w\in \N$, $d$ be a prime power factor of $w$, and $J$ be a knot in $S^3$. Suppose that $\zeta_1\neq \zeta_2$ are primitive $d$th roots of unity satisfying
$2w<\sigma_J(\zeta_1)$ and $\sigma_J(\zeta_2)<-2w$.  Then the knot $K_w(J)$ in $S^1\times S^2$ shown in Figure~\ref{fig:JforNonSlice} is not topologically slice.
\end{proposition}
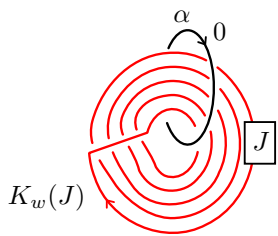
\begin{figure}[htb]
\begin{tikzpicture}
\node at (0,0) {\includegraphics[width=2.5cm]{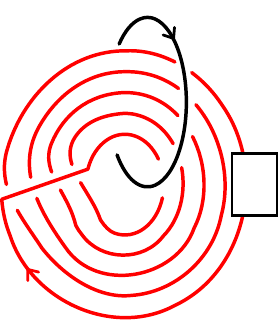}};
\node at (1.05, -0.15) {$J$};
\node at (0.5, 1.3) {$0$};
\node at (0,1.5) {$\alpha$};
\node at (-1.8,-0.9) {$K_w(J)$};
\end{tikzpicture}
\caption{The knot~$K_w(J)$ in $S^1 \times S^2$ with winding number $w$. The specific case $w=5$ is shown.}\label{fig:JforNonSlice}
\end{figure}

\begin{proof}[Proof of Theorem~\ref{thm:thmC}]
In Theorem~\ref{thm:non-slice-even} we saw that there exist infinite families of knots with any given even winding number that are not topologically slice. We have also noted that since we can reverse orientations, it suffices to produce examples for positive winding numbers. Thus, we only need to address the case of odd winding numbers $\geq 5$.

Assume $w\geq 5$.  Thus, $w$ has a factor $d$ which is a prime power with $d\ge 5$. Then  $\zeta_1 = e^{2\pi i/d}$ and $\zeta_2 = e^{4\pi i/d}$ are both primitive $d$th roots of unity.  Moreover, since $d\ge 5$, both $\zeta_1$ and $\zeta_2$ have positive imaginary parts. From~\cite[Theorem 1]{Liv04}, we know that the signature function evaluated at unit complex numbers with positive imaginary parts gives linearly independent maps from the set of knots to $\Z$.  Thus, there exist knots $J'$ and $J''$ in~$S^3$, distinct in topological concordance, such that $2w<\sigma_{J'}(\zeta_1)$, $\sigma_{J'}(\zeta_2)=0$, $\sigma_{J''}(\zeta_2)<-2w$, and $\sigma_{J''}(\zeta_1)=0$. For $i\in \N$, let $J_i = i(J'\#J'')$ be the connected sum of $i$ copies of $J'\#J''$.  By the additivity of the signature function, $J_i$ satisfies the hypotheses of Proposition~\ref{prop:nonslice} and thus, the knot $K_w(J_i)$ is not topologically slice.
Further, using the signature function we see that for $i \neq j$, $J_i \#rJ_i$ and $J_j \# rJ_j$ do not cobound a genus one locally flat orientable surface in $S^3 \times I$.

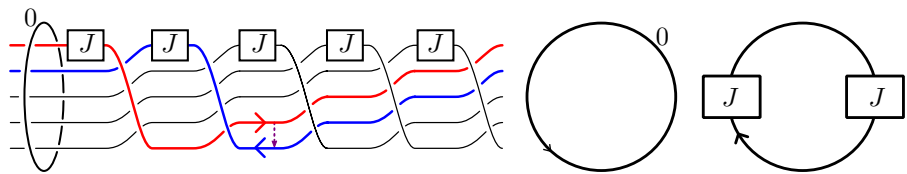
\begin{figure}[htb]
\begin{tikzpicture}
\node at (0,0) {\includegraphics[height=2.0cm]{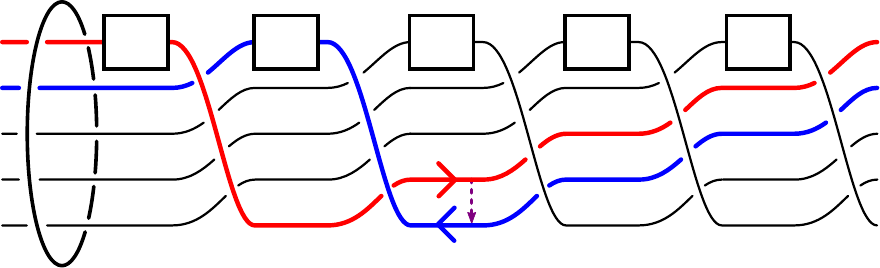}};
\node at (-3, 1.05) {$0$};
\node at (-2.25, 0.71) {$J$};
\node at (-1.15, 0.71) {$J$};
\node at (.04, 0.71) {$J$};
\node at (1.2, 0.71) {$J$};
\node at (2.38, 0.71) {$J$};
\node at (6.1,0) {\includegraphics[height=2.0cm]{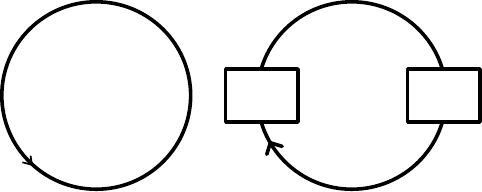}};
\node at (5.4, 0.8) {$0$};
\node at (6.3, 0) {$J$};
\node at (8.25, 0) {$J$};
\end{tikzpicture}
\caption{Left: The $5$-fold cover
of $K_5(J)$, $\widetilde{K_5(J)}$, is a $5$-component link.
Two components of this link are highlighted. Right:  The distant knot $\Phi(J\#rJ)$ obtained from $\widetilde{K_5(J)}$ by a band sum shown on the left.}\label{fig:infnonsliceodd}
\end{figure}
Now we prove that $K_w(J_i)$ is not topologically concordant to $K_w(J_j)$, if $i \neq j$. Suppose we have a topological concordance $C$ between the two knots. The preimage of $C$ in the $w$-fold cover of $S^1\times S^2\times I$ is a disjoint union of topological concordances between the components of the $w$-component links $\widetilde{K_w(J_i)}$ and $\widetilde{K_w(J_j)}$, the lifts of $K_w(J_i)$ and $K_w(J_j)$. Band two adjacent components of $\widetilde{K_w(J_i)}$ as shown in Figure~\ref{fig:infnonsliceodd} together to get a cobordism to the distant knot $\Phi(J_i\# rJ_i)$. Doing the same to $\widetilde{K_w(J_j)}$ yields a locally flat genus one cobordism between the distant knots $\Phi(J_i\#rJ_i)$ and $\Phi(J_j\#rJ_j)$. By Theorem~\ref{thm:distant-knots-cobordisms} we conclude that $J_i\#rJ_i$ and $J_i\#rJ_i$ cobound a locally flat genus one surface in $S^3 \times I$, which is a contradiction.
\end{proof}

In the remainder of this section, we prove Proposition~\ref{prop:nonslice} by developing an obstruction to sliceness in $D^2\times S^2$. The general idea is that if a knot $K$ in $S^1\times S^2$ bounds a locally flat disk in $D^2\times S^2$, then the complement of such a slice disk is a $4$-manifold bounded by some surgery on $K$. Since $D^2\times S^2$ embeds in $S^4$, so does this slice disk exterior and its boundary. We then invoke an obstruction to such an embedding in $S^4$~\cite{Gilmer83} in terms of Casson-Gordon invariants~\cite{Casson86}. Approximating these invariants in the special case of the knot~$K_w(J)$ will finish the proof of Proposition~\ref{prop:nonslice}.

\begin{figure}[htb]
\begin{tikzpicture}
\node at (0,0) {\includegraphics[width=6cm]{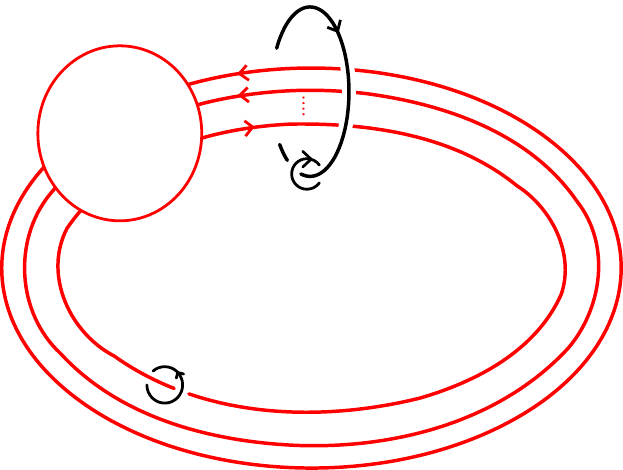}};
\node at (-1.85,1) {\Huge $P$};
\node at (0.5,2) {$0$};
\node at (0,2.5) {$\alpha$};
\node at (0,0.1) {$H$};
\node at (-1.8,-0.9) {$K$};
\node at (-0.9,-1.3) {$\mu_K$};
\end{tikzpicture}
\caption{A generic knot $K$ in $S^1 \times S^2$, with some winding number $w>0$, given as the closure of a tangle $P$.}\label{fig:homologycalculation}
\end{figure}
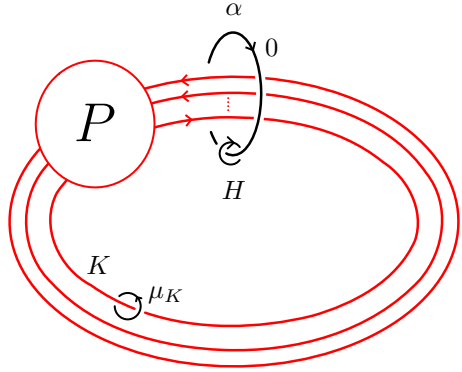
Recall that we draw our knots in the standard diagram for $S^1\times S^2$, as in Figure~\ref{fig:homologycalculation}. For the purposes of the upcoming proof we refer to the $0$-framed curve as $\alpha$.  A positive meridian for $\alpha$ is the Hopf knot, called $H$. Henceforth, an integer framing for a knot $K$ drawn in this fixed standard diagram for $S^1\times S^2$ is the framing with respect to the Seifert framing of $K$ when we consider the surgery diagram as a link $K\sqcup \alpha$ in $S^3$.  A key complication in our analysis is the fact that, since there is no natural $0$-framing of a non-nullhomologous knot in $S^1\times S^2$, we cannot hope that a framing on a slice disk for $K$ in $D^2\times S^2$ will induce the zero framing on this fixed diagram for $K$.  Moreover, the Hopf knot~$H$ need not be nullhomologous in the slice disk exterior. The following proposition reveals that these two difficulties are related.

\begin{proposition}\label{prop:framingrelationship}
Let $K$ be a knot in $S^1\times S^2$ with winding number $w\in \Z$, drawn in the standard diagram for $S^1\times S^2$, as in Figure~\ref{fig:homologycalculation}. Let $\mu_K$ be the positive meridian of $K$, $\Delta$ a topological slice disk for $K$ in $D^2\times S^2$, and $W$ the exterior of $\Delta$ in $D^2\times S^2$. Let $f$ be the framing on $K$ induced by the normal bundle on $\Delta$. Then there exists an integer $\lambda$ such that $[H]=\lambda [\mu_K]$ in $H_1(W;\Z)$ and we have the relationship
\[ f=-2\cdot w\cdot \lambda. \]
\end{proposition}

\begin{proof}
Let $D$ be a slice disk for $H$ given by $D^2\times \{pt\}  \subset D^2 \times S^2$. Suppose that $D$ intersects $\Delta$ in $\lambda^+$ points with positive orientation and $\lambda^-$ points with negative orientation. Removing small neighborhoods of these points from $D$ results in a planar cobordism in $W$ from $H$ to $\lambda^+$ positively oriented meridians of $\Delta$ and $\lambda^-$ negatively oriented meridians.  Thus, $[H]=\lambda [\mu_K]$ in $H_1(W;\Z)$, where $\lambda=\lambda^+-\lambda^-$.

On the other hand, removing neighborhoods of these same points from $\Delta$ turns~$\Delta$ into a planar cobordism from $K$ to $\lambda^+$ positively oriented meridians of $D$ and $\lambda^-$ negatively oriented meridians. This cobordism is contained in the exterior of $D$. Note that this exterior is $D^2\times S^2 - D^2\times Int(D^2)=D^2\times D^2\cong D^4$. Moreover, each of these meridians of $D$ results in a $0$-framed pushoff of $\alpha$. We thus see a planar cobordism $P\subset D^4$ from $K\subset S^3$ to algebraically $\lambda$ copies of $0$-framed pushoffs of $\alpha$.  Since $P\subset \Delta$, it also induces the $f$-framing on $K$. Since $P$ is a surface in $D^4$ bounded by a link in $S^3$, we know that the sum of the framings of the boundary components of $P$ together with the sum of the pairwise linking numbers of the boundary components must be equal to $P\cdot P=0$, that is,
\[
0=f+\lambda^+ \lnk(K,\alpha)-\lambda^- \lnk(K,\alpha)+\lambda^+ \lnk(\alpha,K)-\lambda^- \lnk(\alpha,K)  =f+2w\lambda
\]
and thus $f=-2w\lambda$ as claimed.
\end{proof}

\begin{proposition}\label{prop:HomologyExterior}
Let $K$  be a knot in $S^1 \times S^2$ with winding number~$w \neq 0$ drawn in the standard diagram for $S^1\times S^2$ as shown in Figure~\ref{fig:homologycalculation}. Suppose that $K$ is topologically slice with a slice disk $\Delta\subseteq D^2\times S^2$. Let $W$ be the exterior of $\Delta$ in $D^2\times S^2$. Then the nonvanishing integral homology groups of $W$ and $\bdry W$ are
\begin{align*}
H_0(W)&\cong \Z, & H_1(W)&\cong \Z/w  \\
H_0(\bdry W)&\cong \Z &  H_1(\bdry W)&\cong (\Z/w)^2  & H_3(\bdry W)&\cong\Z
\end{align*}
Furthermore, $H_1(W)$ is generated by the positive  meridian~$\mu_K$ of $K$, and $H_1(\bdry W)$
is generated by $\mu_K$ and the Hopf knot $H$ shown in Figure~\ref{fig:homologycalculation}.
\end{proposition}

\begin{proof}
Let $\nu \Delta$ be a tubular neighbourhood of $\Delta$; note
that $\nu\Delta$ restricts to a tubular neighbourhood~$\nu K$ of the knot~$K$.

First, we compute the homology groups of $W$. Since $W$ is connected, $H_0(W) \cong \Z$. Consider the Mayer-Vietoris sequence for the decomposition $D^2 \times S^2 = \nu \Delta \cup W$. The boundary map $\partial \colon H_2(D^2 \times S^2) \to H_1(\Delta \times S^1)$
is just $\Z \xrightarrow{\times w} \Z$; thus, its kernel vanishes and the cokernel is $\Z/w$. As a result,~$H_1(W)\cong\Z/w$, $H_2(W) = 0$, and $H_3(W) = 0$. Since $H_1(\Delta \times S^1)$ is generated by the $S^1$-factor which is sent in $W$ to $\mu_K$, we conclude that $\mu_K$ generates $H_1(W)$.

Next we compute the homology of $\bdry W$.  We see that $H_0(\bdry W)\cong H_3(\bdry W)\cong \Z$ since $\bdry W$ is a connected, oriented, compact 3-manifold.
Suppose that the normal bundle of $\Delta$ induces the $f$-framing on $K$. Then $\bdry W$ is the result of $0$-surgery on $\alpha$ and $f$-surgery on $K$.  Note that $\lk(K, \alpha)=w$ and that we can compute $H_1(\bdry W)$ in terms of the linking-framing matrix, as in \cite[Proposition 5.3.11]{gompf-stipsicz}.  Thus,
\[
H_1(\bdry W) = \langle \mu_K, H ~|~ w \mu_k=0, f \mu_K+w H=0\rangle.
\]
By Proposition~\ref{prop:framingrelationship}, $f$ is a multiple of $w$ and thus, $f\mu_K=0$ in $H_1(\partial W)$. As a result, the above presentation gives $H_1(\bdry W)\cong (\Z/w)^2$ with generators $\mu_K$ and $H$, as desired.
\end{proof}

\noindent Recall the following theorem of Gilmer--Livingston.

\begin{proposition}[{\cite[Theorem 2.1]{Gilmer83}}]~\label{prop:Gilmer-Livingstonobst}
Let $d$ be a prime power, i.e.\ $d=q^s$ for some prime $q$ and $s\in\N$. Let $W\subset S^4$ be a rational homology ball, and $\chi\colon H_1(W)\to \Z/d\to \CC^\times$ a character.  Then
\[ |\sigma(\bdry W,\chi)|+|\rho-1-\eta(\bdry W,\chi)|\le \rho \]
where $\rho=\frac{1}{2}\dim_{\Z/q} H_1(\bdry W;\Z/q)$.
\end{proposition}

The quantity $\sigma(\bdry W,\chi)$ is the Casson-Gordon signature and the quantity $\eta(\bdry W,\chi)$ is the Casson-Gordon nullity of the $3$-manifold $\partial W$ with respect to the character $\chi$.

The above theorem yields the following obstruction to a knot in $S^1\times S^2$ being topologically slice in $S^2\times D^2$.

\begin{corollary}\label{cor:Casson-Gordon obstruction}
Let $K$ be a topologically slice knot in $S^1\times S^2$ drawn in the standard diagram for $S^1\times S^2$, with winding number $w\neq 0$. Let $d$ be a prime power, i.e.\ $d=q^s$ for some prime $q$ and $s\in\N$, such that $d$ divides $w$. Suppose that the normal bundle of a slice disk~$\Delta$
induces the framing~$f$ on $K$.  Then for any character~$\chi:H_1(M_f(K))\to \Z/d\to \CC^\times$
sending $[H]-\lambda \cdot [\mu_K]\mapsto 1$ where $f=-2\cdot w\cdot \lambda$,
\[  |\sigma(M_f(K),\chi)|\le 1. \]
\end{corollary}

\begin{proof}
Let $W$ be the exterior of a topological slice disk for $K$. We saw earlier that the $f$-framed surgery on $S^1\times S^2$ along $K$ is the boundary of $W$, i.e.\ $\partial W=M_f(K)$. From Proposition~\ref{prop:HomologyExterior} we know that $W$ is a rational homology ball. Indeed, we know that $\dim_{\Z/q} H_1(\bdry W;\Z/q)=2$ since $d$ divides $w$. Moreover, since $D^2\times S^2$ embeds in $S^4$, so does $W$. We also know from Propositions~\ref{prop:HomologyExterior} and~\ref{prop:framingrelationship} that the inclusion induced map $H_1(\partial W)\rightarrow H_1(W)$ simply sends $[\mu_K]\mapsto [\mu_K]$ and $[H]\mapsto\lambda\cdot [\mu_K]$ where $f=-2w\lambda$. This implies that the character $\chi$ descends to a character $\chi:H_1(W)\rightarrow \Z/d\rightarrow \CC^\times$. Thus, we see that
\[
|\sigma(M_f(K),\chi)|\leq |\sigma(\bdry W,\chi)|+|\eta(\bdry W,\chi)|\le 1.
\qedhere\]
\end{proof}

\noindent We now have all the ingredients for a proof of Proposition~\ref{prop:nonslice}.
\begin{proof}[Proof of Proposition \ref{prop:nonslice}]
Consider the knot $K_w(J)$ from Figure~\ref{fig:JforNonSlice}, with winding number $w>0$. Suppose that it is slice, for the sake of contradiction. Let $f$ be the framing on $K_w(J)$ induced by the normal bundle of a slice disk. We saw in Proposition~\ref{prop:framingrelationship} that $f=-2w\lambda$ for some integer $\lambda$. Let $d$ be a prime power factor of $w$, and $\zeta=e^{2\pi i\frac{p}{d}}$ be some primitive $d$th root of unity. Let $\chi:H_1(M_f(K))\to \Z/d\to \CC^\times$ be a character
sending $[H]-\lambda \cdot [\mu_{K_w(J)}]\mapsto 1$ and $[\mu_{K_w(J)}]\mapsto \zeta$, where $H$ is a positive meridian of $\alpha$ and $\mu_{K_w(J)}$ is a positive meridian of $K_w(J)$. We use a formula for the Casson-Gordon signature due to
Cimasoni and Florens~\cite[Theorem 6.7]{CimasoniFlorens08}. In the special case of our presentation of $M_f(K_w(J))$ as a surgery along the $2$-component link $K_w(J)\sqcup \alpha$ their formula states that for $p$ and $\lambda$ with $(d,p)=(d,\lambda p)=1$,
\begin{equation}\label{eqn:Casson-Gordon via Cimasoni-Florens}
\sigma(M_f(K_w(J)),\chi) =
\sigma_{K_w(J)\sqcup \alpha}(e^{2\pi i \frac{p}{d}},
e^{2\pi i \frac{\lambda p}{d}})
+F(w, f, d, p, [\lambda p])
\end{equation}
Note that we are only assuming that $(d,\lambda)=1$ (since $(d,p)=1$) for now, and that we may assume that $p\in\{1,\dots, d-1\}$ since $\zeta$ is a root of unity.  Above $\sigma_{K_w(J)\sqcup \alpha}$ is the colored signature of the $2$-colored link $K_w(J)\sqcup \alpha$ defined in \cite{CimasoniFlorens08}. The quantity $[\lambda p]$ is simply $\lambda p\mod d$ and thus, $[\lambda p]\in \{1,\dots, d-1\}$. Finally,~\cite[Theorem 6.7]{CimasoniFlorens08} gives a formula for the function $F$ involving terms from the linking-framing matrix for a surgery presentation of the $3$-manifold. In our case, the formula yields
\begin{eqnarray*}F(w,f,d,n_1,n_2) &=& -w+\frac{2}{d^2}\left( w(dn_1+ dn_2-2n_1n_2)+n_2(d-n_2)f \right)
\\&=&-w+
2w\left(\frac{n_1}{d}+ \frac{n_2}{d}-2\frac{n_1}{d}\frac{n_2}{d}\right)
+
2\frac{n_2}{d}\left(1-\frac{n_2}{d}\right)f,
\end{eqnarray*}
where $n_1=p$ and $n_2=[\lambda p]$.
Rather than trying to understand this expression precisely, we simply bound its contribution as follows.
\begin{lemma}\label{lem:controlF}
Let $a, f,d\in \Z$, $d\neq 0$ and $n_1,n_2\in \{1,\dots, d-1\}$.  Then
\[
-|a|+\min(0,f)< F(a,f,d,n_1,n_2) < |a|+\max(0,f)
\]
\end{lemma}
We postpone the proof of the above to the end of the section.

Applying this bound above to equation~\eqref{eqn:Casson-Gordon via Cimasoni-Florens}, we see that
\begin{equation}\label{eqn:formula}
\sigma_{K_w(J)\sqcup \alpha}(\zeta, \zeta^\lambda)-w+\min(0,f) < \sigma(M_f(K_w(J)),\chi) < \sigma_{K_w(J)\sqcup \alpha}(\zeta, \zeta^\lambda) +w +\max(0,f)
\end{equation}
By an additivity result of Cimasoni-Florens~\cite[Proposition~2.12]{CimasoniFlorens08} we have that
\begin{equation}\label{eqn:additive}\sigma_{K_w(J)\sqcup \alpha}(\zeta,\zeta^\lambda) = \sigma_{K_w(U)\sqcup \alpha}(\zeta,\zeta^\lambda)+\sigma_J(\zeta).
\end{equation}
Next, observe that the link $K_w(U)\sqcup \alpha$ in $S^3$ bounds a connected $C$-complex $X$ with the first Betti number $\beta_1(X)=w-1$ (recall that $w>0$). Then, we know from~\cite{CimasoniFlorens08} that $\sigma_{K_w(U)\sqcup \alpha}(\zeta, \zeta^\lambda)$ is the signature of a $(w-1)\times(w-1)$ matrix and thus, $|\sigma_{K_w(U)\sqcup \alpha}(\zeta, \zeta^\lambda)|\le w-1$. Applying this  and the additivity relation~\eqref{eqn:additive} to the inequality~\eqref{eqn:formula}, we see that
\[
\sigma_J(\zeta)+1-2w+\min(0,f)
<
\sigma(M_f(K_w(J)),\chi)
<
\sigma_J(\zeta)-1+2w+\max(0,f).
\]
Recall that from Corollary~\ref{cor:Casson-Gordon obstruction}, we know that $\lvert \sigma(M_f(K_w(J)),\chi)\rvert \leq 1$. As a result, we see that
\[
\sigma_J(\zeta)+1-2w+\min(0,f)<1
\]
and
\[-1<\sigma_J(\zeta)-1+2w+\max(0,f)\]

Solve for $\min(0,f)$ and $\max(0,f)$ to see that, as long as $(d, \lambda)=1$, for any primitive $d$th root of unity $\zeta$,
\[
-2w-\sigma_J(\zeta)<\max(0,f) \text{ and } \min(0,f)<2w-\sigma_J(\zeta).
\]
In particular, for the two distinct primitive $d$th roots of unity, $\zeta_1$ and $\zeta_2$, from the statement of the proposition, we see that $\min(0,f)<2w-\sigma_J(\zeta_1)<0$ and $0<-2w-\sigma_J(\zeta_2)<\max(0,f)$. Thus, $\min(0,f)< 0$ and so $f<0$. Similarly, $\max(0,f)>0$ and so $f>0$, and we have reached a contradiction.

It still remains to address the case of $(\lambda,d)\neq 1$. In order to use the formula \eqref{eqn:Casson-Gordon via Cimasoni-Florens}, we will apply a homeomorphism to the manifold $M_f(K_w(J))$, as follows.

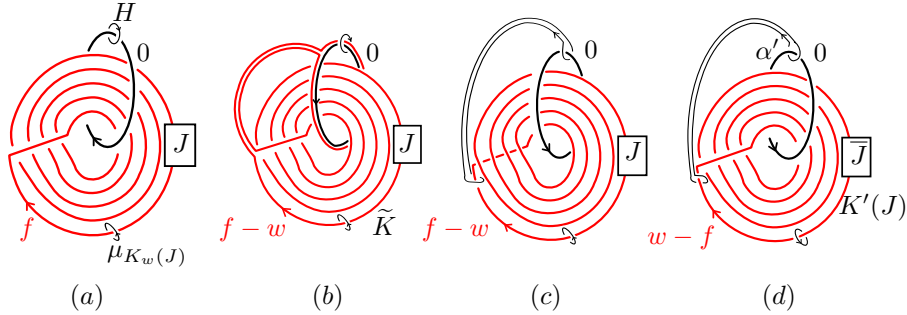
\begin{figure}[htb]
\begin{tikzpicture}
\node at (0,0) {\includegraphics[width=2.5cm]{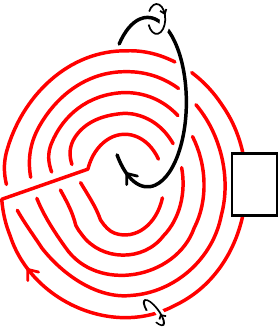}};
\node at (1.05, -0.15) {$J$};
\node at (0.55, 1.05) {$0$};
\node at (-1, -1.3) {\textcolor{red}{$f$}};
\node at (0.3, 1.6) {$H$};
\node at (0.6, -1.65) {$\mu_{K_{w}(J)}$};
\node at (-0.2, -2.2) {$(a)$};
\node at (3,0) {\includegraphics[width=2.5cm]{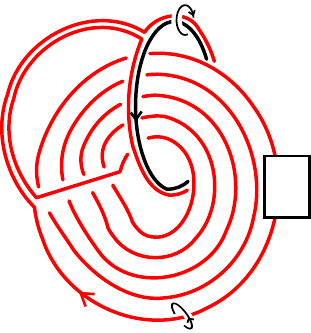}};
\node at (4.05, -0.15) {$J$};
\node at (3.6, 1.05) {$0$};
\node at (2,-1.3) {\textcolor{red}{$f-w$}};
\node at (3.75,-1.2) {$\widetilde{K}$};
\node at (3, -2.2) {$(b)$};
\node at (6,0) {\includegraphics[width=2.5cm]{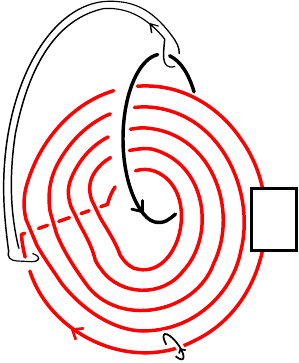}};
\node at (7.06, -0.26) {$J$};
\node at (6.5, 1.1) {$0$};
\node at (4.7,-1.3) {\textcolor{red}{$f-w$}};
\node at (5.9, -2.2) {$(c)$};
\node at (9,0) {\includegraphics[width=2.5cm]{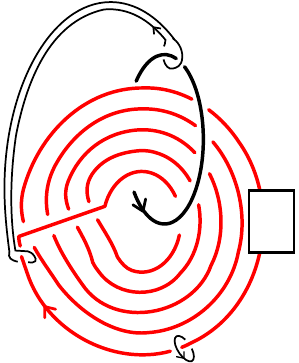}};
\node at (10.06, -0.27) {$\overline{J}$};
\node at (9.55, 1.05) {$0$};
\node at (7.7,-1.4) {\textcolor{red}{$w-f$}};
\node at (9, -2.2) {$(d)$};
\node at (10.25, -1) {$K'(J)$};
\node at (8.85, 1.1) {$\alpha'$};
\end{tikzpicture}
\caption{Left to right:  (a) A surgery diagram for $M_f(K_w(J))$, showing $\mu_K$ and $H$.  (b) A handle slide of $K_w(J)$ over $\alpha$. (c) An isotopy.
(d) The result of the orientation reversing homeomorphism changing every crossing.}
\label{fig:FramingFix}
\end{figure}
We work with the surgery diagram for $M_f(K_w(J))$ in the leftmost panel of Figure~\ref{fig:FramingFix}. Slide $K_w(J)$ over the $0$-framed $2$-handle $\alpha$ as shown in the second panel. The resulting homeomorphism sends the $f$-framing on $K_w(J)$ to the $f-w$ framing on some new knot $\widetilde{K}$ in $S^1\times S^2$ and $M_f(K_w(J))\cong M_{f-w}(\widetilde{K})$. Changing all of the crossings in the diagram gives an orientation reversing homeomorphism from $M_{f-w}(\widetilde{K})$ to a $3$-manifold $M'$. This $3$-manifold is given by surgery on the $2$-component link obtained from $K_w(J)\sqcup \alpha$ by changing all the crossings. In particular, the linking number between the two components is $-w$. We refer to the two components of this link as $K'(\overline{J})$ and $\alpha'$ as shown in Figure~\ref{fig:FramingFix}. Let~$\mu_{K'(\overline{J})}$ and~$H'$  denote the positive meridians of $K'(\overline{J})$ and $\alpha'$ respectively. We just showed that there is an orientation reversing homeomorphism from $M_f(K_w(J))$ to $M'$, such that the induced map on first homology sends $[\mu_{K_{w(J)}}]\mapsto -[\mu_{K'(\overline{J})}]$ and $[H] \mapsto -[H']+[\mu_{K'(\overline{J})}]$. Then the composition of the inverse of this homeomophism with $\chi$, called $\chi'$, sends $[H']\mapsto e^{2\pi i \frac{-p(1+\lambda)}{d}}$ and $[\mu_{K'(\overline{J})}]\mapsto e^{2\pi i \frac{d-p}{d}}$. Orientation reversing homeomorphisms change the  Casson-Gordon invariant by a sign, and thus, $\sigma(M_f(K_w(J),\chi) = -\sigma(M',\chi')$.

Since $d=q^s$ for some prime $q$ and $(d,\lambda)\neq 1$, $q$ must divide $\lambda$. We see that $q$ cannot divide $-p(1+\lambda)$, since then $q$ would divide both $\lambda$ and $\lambda +1$, which leads to a contradiction (note that we have used the fact that $(q,p)=1$). Thus, $(d, -p(1+\lambda))=(d, d-p)=1$ and we can apply the formula of \cite[Theorem 6.7]{CimasoniFlorens08} to see that
\begin{align*}
\sigma(M_f(K_w(J),\chi) &= -\sigma(M',\chi')
\\
&= -\sigma_{K'(\overline{J})\sqcup \alpha'}(e^{2\pi i \frac{d-p}{d}}, e^{2\pi i \frac{-p(1+\lambda)}{d}}) - F(-w,w-f,d,d-p, [-p\cdot(1+\lambda)])
\end{align*}
Above we have used the fact that the linking number for the link $K'(\overline{J})\sqcup \alpha'$ is $-w$. Using the bound from Lemma~\ref{lem:controlF} again, we have that
\[
-w+\min(0,w-f)<F(-w,w-f,d,d-p, [-p\cdot(1+\lambda)]) <w+\max(0,w-f)
\]
Recall that $\zeta=e^{2\pi i\frac{p}{d}}$. Thus, $e^{2\pi i \frac{d-p}{d}}=\overline \zeta$. From formula~\eqref{eqn:additive}, we know that
\begin{align*}
\sigma_{K'(\overline{J})\sqcup\alpha'}(e^{2\pi i \frac{d-p}{d}}, e^{2\pi i \frac{-p(1+\lambda)}{d}})&=\sigma_{K'(U)\sqcup\alpha'}(e^{2\pi i \frac{d-p}{d}}, e^{2\pi i \frac{-p(1+\lambda)}{d}})+\sigma_{\overline J}(e^{2\pi i \frac{d-p}{d}})\\
&=\sigma_{K_w(U)\sqcup \alpha}(e^{2\pi i \frac{d-p}{d}}, e^{2\pi i \frac{-p(1+\lambda)}{d}})+\sigma_{\overline J}(\overline \zeta)
\end{align*}
As before, $K'(U)\sqcup\alpha'$ bounds a $C$-complex $X$ with the first Betti number $\beta_1(X) = w-1$, and thus,
\[ \lvert \sigma_{K'(U)\sqcup \alpha}(e^{2\pi i \frac{d-p}{d}}, e^{2\pi i \frac{p(1-\lambda)}{d}})\rvert \leq w-1 .\] Putting this all together, we see that
\begin{eqnarray*}
-\sigma_{\overline{J}}(\overline\zeta)+1-2w-\max(0,w-f)< \sigma(M_f(K_w(J)),\chi)
\\<-\sigma_{\overline{J}}(\overline\zeta)-1+2w-\min(0,w-f)
\end{eqnarray*}
Signatures change by a negative sign under mirror images, and the signature function is invariant under $\zeta\mapsto \overline\zeta$. Thus, $-\sigma_{\overline{J}}(\overline\zeta) = \sigma_J(\overline\zeta) = \sigma_J(\zeta)$. We may now proceed exactly as we did when $(\lambda,d)=1$. By Corollary~\ref{cor:Casson-Gordon obstruction}, $\lvert \sigma(M_f(K_w(J),\chi)\rvert\leq 1$. Thus, 
\[
\sigma_{J}(\zeta)+1-2w-\max(0,w-f)< \sigma(M_f(K_w(J)),\chi)<1
\]
and
\[
-1<\sigma(M_f(K_w(J),\chi)<\sigma_{J}(\zeta)-1+2w-\min(0,w-f).
\]
For the specific values of $\zeta_1$ and $\zeta_2$ from the statement of the proposition,
\[
\max(0,w-f)>\sigma_J(\zeta_1)-2w>0
\]
and
\[
\min(0,w-f)<\sigma_J(\zeta_2)+2w<0
\]
Since $\max(0,w-f)>0$, $w-f>0$, and since $\min(0,w-f)<0$, $w-f<0$; we have reached a contradiction.
\end{proof}

\begin{proof}[Proof of Lemma~\ref{lem:controlF}]
Suppose for the moment that $a>0$. Since $0<\frac{n_1}{d}<1$ and $0<\frac{n_2}{d}<1$, we see that $-\frac{\sqrt{2}}{2}<\sqrt{2}(\frac{n_1}{d}-\frac{1}{2})<\frac{\sqrt{2}}{2}$ and $-\frac{\sqrt{2}}{2}<\sqrt{2}(\frac{n_2}{d}-\frac{1}{2})<\frac{\sqrt{2}}{2}$. By multiplying these two inequalities together and rearranging we get
\[0<\left(\frac{n_1}{d}+ \frac{n_2}{d}-2\frac{n_1}{d}\frac{n_2}{d}\right)<1
\]
and consequently,
\begin{equation}\label{eq:firsthalf}
0<2a\left(\frac{n_1}{d}+ \frac{n_2}{d}-2\frac{n_1}{d}\frac{n_2}{d}\right)<2a.
\end{equation}
Now note that $0<2\frac{n_2}{d}\left(1-\frac{n_2}{d}\right)$ since it is a product of positive numbers. Moreover, $2\frac{n_2}{d}\left(1-\frac{n_2}{d}\right)<1$, since the function $2x(1-x)-1 = -2\left(x-\frac{1}{2}\right)^2-\frac{1}{2}$ is strictly negative on $(0,1)$. That is, $0<2\frac{n_2}{d}\left(1-\frac{n_2}{d}\right)\leq 1$ and thus,
\begin{equation}\label{eq:secondhalf}
\min(0,f)\leq2\frac{n_2}{d}\left(1-\frac{n_2}{d}\right)f\leq\max(0,f).
\end{equation}
Adding together the inequalities~\eqref{eq:firsthalf} and~\eqref{eq:secondhalf} and subtracting $a$ completes the proof when $a>0$. Now consider the case where $a<0$. As before, we see that \[
0<\left(\frac{n_1}{d}+ \frac{n_2}{d}-2\frac{n_1}{d}\frac{n_2}{d}\right)<1
\]
but since $a$ is negative,
\begin{equation*}\label{eq:firsthalfnegative}
2a<2a\left(\frac{n_1}{d}+ \frac{n_2}{d}-2\frac{n_1}{d}\frac{n_2}{d}\right)<0.
\end{equation*}
Now add the above to the inequality~\eqref{eq:secondhalf} and subtract $a$ from everything to see that
\[
a+\min(0,f)< F(a,f,d,n_1,n_2) < -a+\max(0,f)
\]
where since $a<0$, we have proved that
\[
-|a|+\min(0,f)< F(a,f,d,n_1,n_2) < |a|+\max(0,f)
\]
as needed.
\end{proof}

\section{Smooth vs topological concordance}\label{sec:smooth-vs-top}

In this section we focus on the disparity between the smooth and topological categories for concordance of knots in $S^1\times S^2$, by proving Theorems~\ref{thm:thmBprime} and \ref{thm:thmCprime}. The proofs will be quite similar to those of Theorems~\ref{thm:thmB} and \ref{thm:thmC}, and the main difference will be in the choice of the seed knots $\{J_i\}$ in $S^3$, and the use of smooth invariants to detect non-sliceness, such as Ozsv{\'a}th-Szab{\'o}'s $\tau$ invariant~\cite{OS03}.

\begin{proof}[Proof of Theorem~\ref{thm:thmBprime}]
For $w=0$, let $\{J_i\}$ be an infinite family of topologically slice knots in $S^3$ such that if $i\neq j$ the knots~$J_i\#rJ_i$ and $J_j\#rJ_j$ are not smoothly concordant. For example, let $J_i$ be the connected sum of $i$ copies of the positive clasped untwisted Whitehead double of the right handed trefoil, in which case the $\tau$ invariant detects that $J_i\#rJ_i$ is not smoothly concordant to $J_j\#rJ_j$ for $i \neq j$. We also know that untwisted Whitehead doubles are topologically slice~\cite{Fre82, FQ90, GT04}. Let $K_0(J_i)$ be the winding number zero knot described in Figure~\ref{fig:winding-zero-slice-knots}. Since $J_i$ is topologically slice, each $K_0(J_i)$ is topologically concordant to a distant knot $\Phi(U)$, and thus, the knots $\{K_0(J_i)\}$ are topologically concordant to one another. We saw in the proof of Theorem~\ref{thm:thmB} that $K_0(J_i)$ is slice for any choice of $J_i$. Following the same proof we see that if $K_0(J_i)$ is smoothly concordant to $K_0(J_j)$ in $S^1 \times S^2$ with $i\neq j$  then $J_i\#rJ_i$ is smoothly concordant to $J_j\#rJ_j$ in $S^3$ contradicting our assumption.

Note that it will suffice to prove the theorem for positive winding numbers since examples with negative winding numbers would then be obtained by reversing orientations. For nonzero positive winding numbers, our examples will all be of the form $K_w(J)$, for appropriate choice of knots $J$, as shown in Figure~\ref{fig:many-slice-knots}. We saw in the proof of Theorem~\ref{thm:thmB} that such knots are slice for any choice of $J$.

For $w=2$, let $\{J_i\}$ be an infinite family of topologically slice knots in $S^3$ such that if $i\neq j$ the knots~$(J_i \# J_i)_{2,1}$ and $(J_j \# J_j)_{2,1}$ do not cobound a smooth genus three orientable surface in $S^3 \times I$. For example, let $J_i$ be the connected sum of $i$ copies of the positive Whitehead double of the right handed trefoil, in which case the $\tau$ invariant detects that $(J_i \# J_i)_{2,1}$ and $(J_j \# J_j)_{2,1}$ do not cobound a smooth genus three orientable surface in  $S^3 \times I$. Note that each $K_2(J_i)$ is topologically concordant to the $(2,1)$ cable of the Hopf knot, and thus the knots $\{K_2(J_i)\}$ are topologically concordant to one another. Following the proof of Theorem~\ref{thm:thmB}, we see that if $K_2(J_i)$ are smoothly concordant to $K_2(J_j)$ in $S^1 \times S^2$ then $(J_i \# J_i)_{2,1}$ and  $(J_j \# J_j)_{2,1}$ cobound a smooth genus three orientable surface in $S^3 \times I$, which is a contradiction.

Lastly we address the case~$w>2$. Let $\{J_i\}$ be an infinite family of topologically slice knots in $S^3$ such that for $i\neq j$ the knots~$J_i\#J_i\#(rJ_i)_{2,1}$ and $J_j\#J_j\#(rJ_j)_{2,1}$ do not cobound a smooth genus one orientable surface. Again, one might take $J_i$ to be the connected sum of $i$ copies of the positive Whitehead double of the right handed trefoil, where we use the $\tau$-invariant to verify that $J_i\#J_i\#(rJ_i)_{2,1}$ and $J_j\#J_j\#(rJ_j)_{2,1}$ do not cobound a smooth genus one orientable surface in $S^3 \times I$. Since $J_i$ is topologically slice, $K_w(J_i)$ is topologically concordant to the $(w,1)$ cable of the Hopf knot, and thus the knots $\{K_w(J_i\}$ are topologically concordant to one another. As in the proof of Theorem~\ref{thm:thmB}, if $K_w(J_i)$ and $K_w(J_j)$ are smoothly concordant in $S^1 \times S^2$ then $J_i\#J_i\#(rJ_i)_{2,1}$ and $J_j\#J_j\#(rJ_j)_{2,1}$ cobound a smooth genus one orientable surface in $S^3 \times I$, which is a contradiction.
\end{proof}

\begin{proof}[Proof of Theorem~\ref{thm:thmCprime}] We will use the examples from Proposition~\ref{prop:evencrosscap}. These knots are topologically slice, and moreover are topologically concordant to $\Phi(U)$ when $w=0$ and to the $(2,1)$-cable of $H_{w,1}$ when $w\neq0$. As a result, for a fixed $w$, they are topologically concordant to one another. On the other hand, the knots are not slice since they have non-zero $\gamma_4$. Choose a sub-family of $\{K_i\}$ with distinct smooth non-orientable $4$-ball genera. Since smooth concordance in $S^1\times S^2$ preserves $\gamma_4$, the knots in this subfamily are distinct in smooth concordance.
\end{proof}


\section{An application}

Theorem \ref{thm:thmAprime} implies the result from Cochran-Franklin-Hedden-Horn~\cite{CFHH13} given below. Recall that a \emph{pattern} $P$ is simply a knot in a solid torus $ST$. Given a knot $K$, the \emph{satellite knot} $P(K)$ is obtained by tying $ST$ into the knot $K$ such that the longitude of $ST$ is mapped to the $0$--framed longitude of $K$. The knot $\widetilde{P}$ is the result of applying $P$ to the unknot. The \emph{winding number} of $P$ is the algebraic number of times that $P$ wraps around the longitude of $ST$. Let $M_K$ denote the $0$-framed surgery on a knot $K$. For two knots $K$ and $J$ and $m\neq 0$, we say that $M_K$ is $\Z[1/m]$-homology cobordant to $M_J$ \emph{rel meridians} if $M_K$ and $M_J$ are $\Z[1/m]$-homology cobordant via a $4$-manifold $W$ such that the positively oriented meridians of $K$ and $J$ differ in $H_1(W;\Z[1/m]) $ by a positive unit in $\Z[1/m]$.

\begin{corollary}[{\cite[Theorem 2.1]{CFHH13}}] If $P$ is a pattern with winding number $w> 0$ such that $\widetilde{P}$ bounds a smooth disk in a $\Z\left[\frac{1}{w}\right]$-homology ball, then for any knot $K$, $M_K$ is $\Z\left[\frac{1}{w}\right]$-homology cobordant to $M_{P(K)}$ rel meridians.
\end{corollary}
\begin{proof}
Let $ST$ denote the solid torus containing $P$, and let $\alpha$ be the meridian of $ST$. The zero framed surgery on $P(K)$ can be built by gluing the result of zero framed surgery on $P$ in $ST$ to $S^3-K$ so that $\alpha$ is identified with the meridian of $K$ and the longitude of $K$ is identified with the longitude of $ST$.

Since $\widetilde{P}$ bounds a smooth disk in a $\Z\left[\frac{1}{w}\right]$-homology ball, $M_{\widetilde{P}}$ is a $\Z\left[\frac{1}{w}\right]$-homology $S^1\times S^2$ which bounds a $\Z\left[\frac{1}{w}\right]$-homology $S^1\times D^3$.  Thus, by Proposition~\ref{prop:knots-in-homology-S1xS2}, there is a $\Z\left[\frac{1}{w}\right]$-homology cobordism $W$ from $M_{\widetilde{P}}$ to $S^1\times S^2$ in which $\alpha$ is concordant to the Hopf knot. Let $C$ be such a concordance. Cut out a neighborhood of $C$ from $W$ and glue in $S^3-K \times I$ so that we obtain a $\Z\left[\frac{1}{w}\right]$-homology cobordism from $M_{P(K)}$ to $M_K$. Note that the `fat' meridian of $K$ in $M_{P(K)}$ (i.e.\ the image of $\alpha$) is concordant in the $4$-manifold to the meridian of $K$ in $M_K$. In $M_{P(K)}$, $\alpha$ is homologous to $m$ times the meridian of $P(K)$.
\end{proof}

\bibliographystyle{alpha}
\bibliography{bib}

\end{document}